\documentclass[11pt]{amsart}

\usepackage[utf8]{inputenc}
\usepackage{amsmath}
\usepackage{amsfonts}
\usepackage{amssymb}
\usepackage{latexsym}
\usepackage{graphicx}
\usepackage{longtable} 
\usepackage{multirow} 
\usepackage{wrapfig}
\usepackage{color}
\usepackage{mathtools}
\usepackage{ytableau}
\usepackage{verbatim}
\usepackage[all]{xy}
\usepackage{graphics}
\usepackage{pdfsync}
\usepackage[backref=page]{hyperref}

\theoremstyle{plain}
\newtheorem{theorem}{Theorem}[section]
\newtheorem{proposition}[theorem]{Proposition}

\newtheorem{lemma}[theorem]{Lemma}
\newtheorem{corollary}[theorem]{Corollary}

\theoremstyle{definition}
\newtheorem{definition}[theorem]{Definition}
\newtheorem{example}[theorem]{Example}

\theoremstyle{remark}
\newtheorem{remark}[theorem]{Remark}

\numberwithin{equation}{section}
\newcommand{\bC}{\mathbb{C}}

\newcommand{\bQ}{\mathbb{Q}}
\newcommand{\bR}{\mathbb{R}}
\newcommand{\bZ}{\mathbb{Z}}

\newcommand{\cD}{\mathcal{D}}

\newcommand{\str}{\operatorname{Str}}

\newcommand{\des}{\mathrm{Des}} 

\newcommand{\suchthat}{\;|\;}

\newcommand{\dsc}{\mathrm{des}} 

\newcommand{\sym}{\mathrm{Sym}} 

\newcommand{\anti}{\operatorname{\mathrm{Anti}}}
\newcommand{\schub}{\mathfrak{S}}

\newcommand{\ds}[2]{\big\langle {#1}\big\rangle_{#2}}

\newcommand{\alpx}{\mathbf{x}}
\newcommand{\pipe}{\mathrm{PD}}

\newcommand{\wc}[2]{\mathcal{W}_{#1}^{(#2)}}\newcommand{\wcp}[1]{\mathcal{W}_{#1}^{'}}
\newcommand{\Lukwc}[1]{\mathcal{LC}_{#1}} 

\newcommand{\code}[1]{\mathrm{code}(#1)} 

\newcommand{\Luk}{\mathcal{LP}} 
\newcommand{\Cox}{\mathrm{Cox}}

\newcommand{\sgrp}{\mathrm{S}}
\newcommand{\sgrpp}{\mathrm{S}'}

\newcommand{\reduced}{\mathrm{Red}}

\newcommand{\flag}[1]{\mathrm{Flag}({#1})}
\newcommand{\Hess}{\mathcal{H}}
\newcommand{\Perm}{\mathrm{Perm}} 
\newcommand{\Pet}{\mathrm{Pet}} 
\newcommand{\class}{\tau}
\newcommand{\M}{\mathrm{M}}
\newcommand{\Mt}{\widetilde{\mathrm{M}}}
\newcommand{\D}{\mathcal{K}}
\newcommand{\divdiff}{\mathfrak{d}}

\usepackage[normalem]{ulem}

\begin{document}

\title[The permutahedral variety, mixed Eulerian numbers]{The permutahedral variety, mixed Eulerian numbers, and principal specializations of Schubert polynomials}

\author{Philippe Nadeau}
\address{Univ Lyon, CNRS, Universit\'e Claude Bernard Lyon 1, UMR
5208, Institut Camille Jordan, 43 blvd. du 11 novembre 1918, F-69622 Villeurbanne cedex, France}
\email{\href{mailto:nadeau@math.univ-lyon1.fr}{nadeau@math.univ-lyon1.fr}}

\author{Vasu Tewari}
\address{Department of Mathematics, University of Hawaii at Manoa, Honolulu, HI 96822, USA}
\email{\href{mailto:vvtewari@math.hawaii.edu}{vvtewari@math.hawaii.edu}}



\begin{abstract}
	We compute the expansion of the cohomology class of the permutahedral variety in the basis of Schubert classes.  The resulting structure constants $a_w$ are expressed as a sum of \emph{normalized} mixed Eulerian numbers indexed naturally by reduced words of $w$.
The description implies that the $a_w$ are positive for all permutations $w\in S_n$ of length $n-1$, thereby answering a question of Harada, Horiguchi, Masuda and Park. We use the same expression to establish the invariance of $a_w$ under taking inverses and conjugation by the longest word, and subsequently establish an intriguing cyclic sum rule for the numbers.

We then move toward a deeper combinatorial understanding for the $a_w$ by exploiting in addition the relation to Postnikov's divided symmetrization. Finally, we are able to give a combinatorial interpretation for $a_w$ when $w$ is vexillary, in terms of certain tableau descents. It is based in part on a relation between the numbers $a_w$ and principal specializations of Schubert polynomials.

Along the way, we prove results and raise questions of independent interest about the combinatorics of permutations, Schubert polynomials and related objects. We also sketch how to extend our approach to other Lie types, highlighting in particular an identity of Klyachko.
\end{abstract}

\maketitle
\setcounter{tocdepth}{1}
\tableofcontents

\section{Introduction and statement of results}
\subsection{Background}
The (type $A$) complete flag variety $\flag{n}$ has been an active area of study for many decades.
In spite of its purely geometric origins, it interacts substantially with representation theory  and algebraic combinatorics.
By way of the intricate combinatorics involved in the study of its Schubert subvarieties, the study of $\flag{n}$ poses numerous intriguing questions.
The bridge between the geometry and topology of Schubert varieties and the associated algebra and combinatorics is formed in great part by Schubert polynomials, relying upon seminal work of Borel \cite{Bor53} and Lascoux-Sch\"utzenberger \cite{Las82}, followed by influential work of Billey-Jockusch-Stanley \cite{Bil93} and Fomin-Stanley \cite{FS94}.
A fundamental open problem at the intersection of algebraic combinatorics and enumerative algebraic geometry is that of finding a combinatorial rule for structure constants $c_{uv}^w$ arising in the product of Schubert polynomials $\mathfrak{S}_u\mathfrak{S}_v=\sum_{w}c_{uv}^w\mathfrak{S}_w$.
Geometrically, these constants encode certain intersection numbers of Schubert varieties.
We refer to them as the \emph{generalized Littlewood-Richardson (LR) coefficients} henceforth.

Hessenberg varieties are a relatively recent family of subvarieties of $\flag{n}$ introduced by De Mari, Procesi and Shayman \cite{MPS92} with inspiration from numerical analysis.
Their study has also revealed a rich interplay between geometry, representation theory and combinatorics \cite{And10,HT17,Ty07}, and the last decade has witnessed an ever-increasing interest with impetus coming from the study of chromatic quasisymmetric functions and its ramifications for the Stanley-Stembridge conjecture \cite{HP18, ShWa12,ShWa16}.
The study of the cohomology rings of Hessenberg varieties has been linked to the study of hyperplane arrangements and representations of the symmetric group \cite{HTM19,AbHo16,BC18,HarHor18}.
We refer the reader to Abe and Horiguchi's excellent survey article \cite{AbeHor19} and references therein for more details on the rich vein of mathematics surrounding Hessenberg varieties.

 To define a Hessenberg variety $\mathcal{H}(X, h)$ in $\flag{n}$, one needs an $n \times n$ matrix $X$ and a \emph{Hessenberg function} $h :[n] \to [n]$.
We fix $h$ to be $(2,3,\dots,n,n)$.
The \emph{permutahedral variety} $\Perm_n$ is the \emph{regular semisimple}  Hessenberg variety corresponding to this choice of $h$ and $X$ being a diagonal matrix with distinct entries along the diagonal.
This variety is a smooth toric variety whose fan comprises the Weyl chambers of the type A root system.
It appears in many areas in mathematics \cite{MaSh88,Kly85, Pro90}, and notably is a key player in the Huh-Katz resolution of the Rota-Welsh conjecture in the representable case \cite{HK12}. The \emph{Peterson variety} $\Pet_n$ is the \emph{regular nilpotent} Hessenberg variety defined with the same $h$, and with $X$ chosen to be the nilpotent matrix that has ones on the upper diagonal and zeros elsewhere. This variety has also garnered plenty of attention recently; see \cite{Ba17, Dr15, HaTy11, In15, IT16, IY12, Ri03}.

 It is known that for a given $h$, all {regular} Hessenberg varieties have the same class in the rational cohomology $H^*(\flag{n})$, see \cite{AFZ20_bis}.
We let $\class_n$ be this cohomology class for $h=(2,3,\dots,n,n)$, so we have $\class_n=[\Perm_n]=[\Pet_n]$.  Since $\Perm_n$ and $\Pet_n$ are irreducible subvarieties of $\flag{n}$ of complex dimension $n-1$, the class $\class_n$ lives in degree $(n-1)(n-2)$, and we may consider its Schubert class expansion
\begin{equation}
\label{eq:peterson_class_decomposition}
 \class_n=\sum_{w\in \sgrpp_n}a_w\sigma_{w_ow},
\end{equation}
where $\sgrpp_n$ denotes the set of permutations in $\sgrp_n$ of length $n-1$.
Given the geometric interpretation for the $a_w$ as certain intersection numbers, it follows that $a_w\in \bZ_{\geq 0}$.
\smallskip

\subsection{Motivation}
The main goal of this article is to develop a concrete understanding of the coefficients $a_w$ in \eqref{eq:peterson_class_decomposition}.
To put our results in context, we recall what earlier results say about these coefficients.
In fact, Anderson and Tymoczko \cite{And10} give an expansion for $[\mathcal{H}(X,h)]$ for arbitrary $h$ which involves multiplication of Schubert polynomials depending on \emph{length-additive} factorizations of a permutation $w_h$ attached to $h$; see Subsection~\ref{subsec:approach_AT} for details.
In general, transforming this expression into one in the basis of Schubert polynomials in a combinatorially explicit manner would require understanding generalized LR coefficients.
In fact, the special cases in which Anderson and Tymoczko provide explicit expansions in terms of Schubert polynomials are those for which combinatorial rules are indeed known \cite[Sections 5 and 6]{And10}.

The case of $\tau_n$ appears again in work of Harada et al \cite[Section 6]{HarHor18} as well as Kim \cite{Kim20}. In the former, $\tau_n$ is expressed as a sum of classes of Richardson varieties \cite[Theorem 6.4]{HarHor18}.
Yet again, translating this into an explicit expansion in terms of Schubert classes amounts to understanding certain generalized LR coefficients.

In light of this discussion, we are led to approach the question of providing a meaningful perspective on the $a_w$, and thereby $\tau_n$, via alternative means.
To this end we bring together work of Klyachko \cite{Kly85,Kly95} and Postnikov \cite{Pos09}, and explicitly describe the $a_w$ as certain sums of mixed volumes of hypersimplices.
In so doing, we unearth interesting connections between these numbers and the combinatorics of reduced words, principal specializations of Schubert polynomials, and enumeration of flagged tableaux.
\emph{Our work also brings forth certain properties of the $a_w$ that we do not know geometric reasons for}.
Furthermore, since we bypass the computation of generalized LR coefficients, our analysis of the $a_w$ sheds light on various relations that are imposed between the two quantities in question.
It is our hope that understanding classes of other regular Hessenberg varieties can advance our understanding of generalized LR coefficients.

\subsection{Main results}
We proceed to state our main results.
The reader is referred to Section~\ref{sec:preliminaries} for  undefined terminology.
Our first main result states that the $a_w$ are strictly positive, that is, the expansion in \eqref{eq:peterson_class_decomposition} has \emph{full support}.
This answers a problem posed by Harada et al \cite[Problem 6.6]{HarHor18}.
\begin{theorem}\label{thm:main_1}
	For $w\in \sgrpp_n$, we have that $a_w>0$ from the explicit formula
	\[
	a_w=\frac{1}{(n-1)!}\sum_{\mathbf{i}\in \mathrm{Red}(w)}A_{c(\mathbf{i})}.
	\]
	Furthermore, the following symmetries hold.
	\begin{itemize}
	\item $a_w=a_{w_oww_o}$ where $w_o$ denotes the longest word in $\sgrp_n$.
	\item $a_{w}=a_{w^{-1}}$.
	\end{itemize}
\end{theorem}
This theorem is the succinct version of the contents of Proposition~\ref{prop:stability}, Theorem~\ref{thm:positive_formula}, and Corollary~\ref{cor:aw_positivity}. Here $\mathrm{Red}(w)$ denotes the set of reduced words of $w$ and the $A_{c(\mathbf{i})}$ are certain mixed Eulerian numbers indexed by weak compositions $c(\mathbf{i})$ determined by reduced words for $w$.
These numbers were introduced by Postnikov \cite[Section 16]{Pos09} as mixed volumes of Minkowski sums of hypersimplices, and they generalize the classical Eulerian numbers.
Curiously, while geometry tells us that the $a_w$ are nonnegative integers, our formula expresses them as a sum of positive rational numbers.
That this sum is indeed integral hints at deeper reasons, which is what we explore subsequently.

Any permutation has a natural factorization into indecomposable permutations acting on disjoint intervals, where $u\in \sgrp_p$ is called indecomposable if the image of $[i]$ does not equal $[i]$ for $i=1,\ldots,p-1$; see Section~\ref{sub:indecomposable} for precise definitions. One may rotate such blocks, thus giving rise to cyclic shifts of the permutation $w$. Given $w\in \sgrpp_n$, let $w=w^{(1)}, w^{(2)},\ldots,w^{(k)}$ be its cyclic shifts.

Our next chief result is a \emph{cyclic sum rule}:
\begin{theorem}\label{thm:main_2}
	For $w\in \sgrpp_n$ and with the notation just established we have that
	\[
		\sum_{1\leq i\leq k}a_{w^{(i)}}=|\mathrm{Red}(w)|.
	\]
\end{theorem}
This theorem is stated as Theorem~\ref{thm:cyclic_sum} in Section~\ref{sec:properties}.
Again, the appearance of the number of reduced words on the right hand side is mysterious from a geometric perspective.
Furthermore, \emph{what explains the seemingly ad hoc appearance of the cyclic rotations of block factorizations in this context?}
Theorem~\ref{thm:main_2} hints at a potential refinement of the set of reduced words of $w$ that would provide a combinatorial interpretation to the $a_{w}$.
While we do not have such an interpretation in general, we obtain interpretations for important classes of permutations; we describe our results next.

Divided symmetrization is a linear form which acts on the space of polynomials in $n$ indeterminates of degree $n-1$.
This was introduced by Postnikov \cite{Pos09} in the context of computing volume polynomials of permutahedra.
In its most general form, this operator sends a polynomial $f(x_1,\dots,x_n)$ to a symmetric polynomial $\ds{f(x_1,\dots,x_n)}{n}$ as follows:
\begin{align}\label{eq:definition_ds}
  \ds{f(x_1,\dots,x_n)}{n}\coloneqq \sum_{w\in \sgrp_n}w\cdot\left(\frac{f(x_1,\dots,x_n)}{\prod_{1\leq i\leq n-1}(x_i-x_{i+1})}\right),
\end{align}
where $S_n$  acts by permuting variables.
For homogeneous $f$ of degree $n-1$, its divided symmetrization $\ds{f}{n}$ is a scalar, and it is in this context where our results are primarily set.
A  computation starting with the Anderson-Tymoczko class of the Peterson variety \cite{And10} leads us to the following conclusion already alluded to in the prequel \cite{DS} to this article  \textemdash{} \emph{for $w\in \sgrpp_n$, we have that $a_w=\ds{\mathfrak{S}_w}{n}$.}
We are thus able to leverage our earlier work to obtain a better handle on the $a_w$.
\smallskip

We introduce a class of permutations in $\sgrpp_n$ for which the corresponding $a_w$ are particularly nice.
We refer to these permutations as \L{}ukasiewicz permutations in view of how they are defined.
The set of \L{}ukasiewicz permutations has cardinality given by the $(n-1)$-th Catalan number.
A characteristic feature of these permutations is that a Schubert polynomial indexed by any such permutation is a sum of Catalan monomials (see \cite{DS}), and thus we have our next result.
\begin{theorem}\label{thm:main_3}
	For $w\in \mathcal{LP}_n$, we have that
	\[
	a_w=\mathfrak{S}_w(1,\dots,1).
	\]
	In particular, $a_w$ equals the number of reduced pipe dreams for any \L{}ukasiewicz permutation $w\in \sgrpp_n$.
\end{theorem}
In particular it follows that for $132$-avoiding and $213$-avoiding permutations $w\in \sgrpp_n$, we have that $a_w=1$. Another special case concerns Coxeter elements, for which $\mathfrak{S}_w(1,\dots,1)$ can be expressed as the number of permutations in $\sgrp_{n-1}$ with a given descent set depending on $w$.
Theorem~\ref{thm:main_3} is stated as Theorem~\ref{thm:Lukasiewicz_permutations}.

Our final results concern the important class of permutations known as vexillary permutations, starting with the larger class of quasiindecomposable permutations.
To state our results we need some more notation.
Permutations of the form $1^a\times u\times 1^b$ for $u$ indecomposable and $a,b\geq 0$,  are said to be \emph{quasiindecomposable}. Here $1^a\times u\times 1^b$ denotes the permutation obtained from $u$ by inserting $a$ fixed points at the beginning and $b$ fixed points at the end.

Set $\nu_u(j)\coloneqq \mathfrak{S}_{1^j\times u}(1,1,\dots)$ for $j\geq 0$. The following is presented as Theorem~\ref{thm:quasiindecomposable} later.
\begin{theorem}\label{thm:main_4}
Let $u\in \sgrp_{p+1}$ be an indecomposable permutation of length $n-1$.  We have that
\[
\sum_{j\geq 0}\nu_u(j)t^j=\frac{{\displaystyle{\sum_{m=0}^{n-p-1}a_{1^m\times u\times 1^{n-p-1-m}}t^m}}}{(1-t)^n},
\]
\end{theorem}
We now come to our last result, which is of independent interest, making no mention of the numbers $a_w$.
We establish that in the case where $u$ is a vexillary permutation, the quantity $\nu_u(j)$ is essentially the order polynomial of a model of \emph{$(P,\omega)$-partitions} for appropriately chosen poset $P$ and labeling $\omega$.
We refer the reader to Section~\ref{sec:vexillary} for precise details, wherein the following result is stated as Theorem~\ref{thm:vexillary}.
\begin{theorem}\label{thm:main_5}
Let $u\in \sgrp_{p+1}$ be an indecomposable vexillary permutation with shape $\lambda\vdash n-1$. Then there exist a labeling $\omega_u$ of $\lambda$ and an integer $N_u\geq 0$ such that
\[
\sum_{j\geq 0}\nu_u(j)t^j=\frac{\displaystyle{\sum_{T\in \mathrm{SYT}(\lambda)} t^{\dsc(T;\omega_u)-N_u}}}{(1-t)^n},
\]
where $\mathrm{SYT}(\lambda)$ denotes the set of standard Young tableaux of shape $\lambda$.
\end{theorem}
In conjunction with Theorem~\ref{thm:main_4} above, this theorem yields a combinatorial interpretation for $a_w$ for $w$ vexillary. In the case $u$ is indecomposable Grassmannian (respectively dominant), the statistic $\dsc(T;\omega_u)$  in the statement of Theorem~\ref{thm:main_4} coincides with the usual descent (respectively ascent) statistic on standard Young tableaux for the appropriate choice of $\omega_u$.
\bigskip

\noindent {\bf Outline of the article:} Section ~\ref{sec:preliminaries} provides the necessary background on basic combinatorial notions attached to permutations, the cohomology of the flag variety, and some important properties of Schubert polynomials.
Section~\ref{sec:formulas} provides two perspectives on computing $a_w$, the first via Klyachko's investigation of the rational cohomology ring of $\Perm_n$, and the second via Postnikov's divided symmetrization and a formula due to Anderson and Tymoczko.
Section~\ref{sec:mixed_eulerian} introduces the mixed Eulerian numbers and surveys several of their properties, including a recursion that uniquely characterizes them. It also discusses Petrov's probabilistic take on these numbers.
In Section~\ref{sec:properties}, we use results of the preceding section to establish Theorems~\ref{thm:main_1},~\ref{thm:main_2} and ~\ref{thm:main_4}.
Section~\ref{sec:combinatorial_interpretation} discusses combinatorial interpretations for the $a_w$ in special cases. In particular, we discuss the case of  \L{}ukasiewicz permutations, Coxeter elements as well as Grassmannian permutations, proving \ref{thm:main_3} in particular.
Section~\ref{sec:vexillary} establishes our most general result as far as combinatorial interpretations go, by providing a complete understanding of the $a_w$ for vexillary $w$ through Theorem~\ref{thm:main_5}. Section 8 deals with the problem in general type $\Phi$, and includes Klyachko's reduced word identity for Schubert classes with its application the numbers $a_w^\Phi$.
We conclude with various remarks on further avenues and questions in Section~\ref{sec:further_remarks}.

\section{Preliminaries}
\label{sec:preliminaries}

\subsection{Permutations}
\label{sub:permutation_statistics}

We denote by $\sgrp_n$ the group of permutations of $\{1,\ldots,n\}$. We write an element $w$ of  $\sgrp_n$ in one line notation, that is, as the word $w(1)w(2)\cdots w(n)$. The permutation $w_o=w^n_o$ is the element $n(n-1)\cdots21$.
\smallskip

{\noindent \bf Descents:} An index $1\leq i<n$ is a \emph{descent} of $w\in \sgrp_n$ if $w(i)>w(i+1)$.
The set of such indices is the \emph{descent set} $\des(w)\subseteq [n-1]$ of $w$.
Given a subset $S\subseteq [n-1]$, define $\beta_n(S)$ to be the number of permutations $w\in\sgrp_n$ such that $\des(w)=S$.  If $n=4$ and $S=\{1,3\}$, one has $\beta_4(S)=|\{2143,3142,4132,3241,4231\}|=5$.\smallskip

{\noindent \bf Code and length:} The \emph{code} $\code{w}$ of a permutation $w\in\sgrp_n$ is the sequence $(c_1,c_2,\ldots,c_n)$ given by $c_i=|\{j>i\suchthat w(j)<w(i)\}|$.
 The map $w\mapsto \code{w}$ is a bijection from $\sgrp_n$ to the set  $C_n\coloneqq \{(c_1,c_2,\ldots,c_n)\suchthat 0\leq c_i\leq n-i,~1\leq i\leq n\}$.
 The {\em shape} $\lambda(w)$ is the partition obtained by rearranging the nonzero elements of the code in nonincreasing order.  The \emph{length} $\ell(w)$ of a permutation $w\in\sgrp_n$ is the number of {\em inversions}, i.e. pairs $i<j$ such that $w(i)>w(j)$. It is therefore equal to the sum $\sum_{i=1}^nc_i$ if $(c_1,\ldots,c_n)$ is the code of $w$. The permutation $w=3165274\in \sgrp_7$ has code $c(w)=(2,0,3,2,0,1,0)$, shape $\lambda(w)=(3,2,2,1)$ and length $8$.\smallskip

Let us recall the definition of the set $\sgrpp_n$, which naturally index the coefficients $a_w$:
\begin{equation}
\label{eq:Sn_prime}
\sgrpp_n\coloneqq \{w\in \sgrp_n \suchthat  \ell(w)=n-1\}.
\end{equation}

The cardinality of $\sgrpp_n$ for $n=1,\ldots,10$ is $|\sgrpp_n|=1,1,2,6,20,71,259,961,3606,13640$. The sequence occurs as number A000707 in the Online Encyclopaedia of Integer Sequences \cite{oeis}.\smallskip

{\noindent \bf Pattern avoidance:} Let $u\in \sgrp_k$ and $w\in \sgrp_n$ where $k\leq n$.
An occurrence of the pattern $u$ in $w$ is a sequence $1\leq i_1<\cdots <i_k\leq n$ such that $u(r)<u(s)$ if and only if $w(i_r)<w(i_s)$. We say that $w$ \emph{avoids} the pattern $u$ if it has no occurrence of this pattern and we refer to $w$ as \emph{$u$-avoiding}.
For instance, $35124$ has two occurrences of the pattern $213$ at positions $1<3<5$ and $1<4<5$. It is $321$-avoiding. \smallskip

{\noindent \bf Reduced words:} The symmetric group $\sgrp_n$ is generated by the elementary transpositions $s_i=(i,i+1)$ for $i=1,\ldots,n-1$.  Given $w\in \sgrp_n$, the minimum length of a word $s_{i_1}\cdots s_{i_l}$ in the $s_i$'s representing $w$ is the length $\ell(w)$ defined above, and such a word is called a \emph{reduced expression} for $w$. We denote by $\reduced(w)$ the set of all \emph{reduced words}, where $i_1\cdots i_l$ is a reduced word for $w$ if $s_{i_1}\cdots s_{i_l}$ is a reduced expression of $w$. For the permutation $w=3241$ of length $4$, $\reduced(w)=\{1231,1213,2123\}.$
With these generators, $\sgrp_n$ has a well-known Coxeter presentation given by the relations $s_i^2=1$ for all $i$, $s_is_j=s_js_i$ if $|j-i|>1$ and $s_is_{i+1}s_i=s_{i+1}s_is_{i+1}$ for $i<n-1$. These last two sets of relations are called the \emph{commutation relations} and \emph{braid relations} respectively. Note that $321$-avoiding permutations can be characterized as {\em fully commutative}: any two of their reduced expressions can be linked by a series of commutation relations \cite{Bil93}.\smallskip

{\noindent\bf The limit $\mathbf{\sgrp_\infty}$:} One has natural monomorphisms $\iota_n: \sgrp_n\to \sgrp_{n+1}$ given by adding the fixed point $n+1$.
One can then consider the direct limit of the groups $\sgrp_n$, denoted by $\sgrp_\infty$: it is naturally realized as the set of permutations $w$ of $\{1,2,3,\ldots\}$ such that $\{i\suchthat w(i)\neq i\}$ is finite. Any group $\sgrp_n$ thus injects naturally in $\sgrp_\infty$ by restricting to permutations for which all $i>n$ are fixed points.

Most of the notions we defined for $w\in \sgrp_n$ are well defined for $\sgrp_\infty$. The code can be naturally extended to $w\in\sgrp_\infty$ by defining $c_i=|\{j>i\suchthat w(j)>w(i)\}|$ for all $i\leq 1$. It is then a bijection between $\sgrp_\infty$ and the set of infinite sequences $(c_i)_{i\geq 1}$ such that $\{i\suchthat c_i>0\}$ is finite. The length $\ell(w)$ is thus also well defined. Occurrences of a pattern $u\in\sgrp_k$ are well defined in $\sgrp_\infty$ if $u(k)\neq k$\footnote{This restriction is necessary since for instance $4321$ avoids $213$ but $43215=\iota_4(4321)$ does not}. Reduced words extend naturally.

\subsection{Flag variety, cohomology and Schubert polynomials}
\label{sub:cohomology}

Here we review standard material that can be found for instance in \cite{Fulton,Man01,Bri05} and the references therein.

The flag variety $\flag{n}$ is defined as the set of complete flags $V_\bullet=(V_0=\{0\}\subset V_1\subset V_2 \subset \cdots \subset V_n=\bC^n)$ where $V_i$ is a linear subspace of $\bC^n$ of dimension $i$ for all $i$. For example,  $V^{std}_\bullet,V^{opp}_\bullet$ are the standard and opposite flags given by $V^{std}_i= \operatorname{span}(e_1,\ldots,e_i)$ and $V_i^{opp}=\operatorname{span}(e_{n-i+1},\ldots,e_n)$ respectively.
$\flag{n}$ has a natural structure of a smooth projective variety of dimension $\binom{n}{2}$.
It admits a natural transitive action of $GL_n$ via $g\cdot V_\bullet=(\{0\}\subset g(V_1)\subset g(V_2) \subset \cdots \subset \bC^n)$.
In fact $\flag{n}$ is part of the family of \emph{generalized flag varieties $G/B$}, with $G$ a connected reductive group and $B$ a Borel subgroup.
In this context, $\flag{n}$ corresponds to the type A case, with $G=GL_n$ and $B$ the group of upper triangular matrices.

 Given any fixed reference flag $V_\bullet^{ref}$, $\flag{n}$ has a natural affine paving given by Schubert cells $\Omega_w(V^{ref}_\bullet)$ indexed by permutations $w\in \sgrp_n$. As algebraic varieties one has $\Omega_w(V_\bullet^{ref})\simeq \bC^{\ell(w)}$ where $\ell(w)$ is the length of $w$. By taking closures of these cells, one gets the family of \emph{Schubert varieties} $X_w(V_\bullet^{ref})$.

The cohomology ring $H^*(\flag{n})$ with rational coefficients is a well-studied graded commutative ring that we now go on to describe.
It is known that to any irreducible subvariety $Y\subset \flag{n}$ of dimension $d$ can be associated a {\em fundamental class} $[Y]\in H^{n(n-1)-2d}(\flag{n})$. In particular there are classes $[X_{w}(V_\bullet^{ref})]\in H^{n(n-1)-2\ell(w)}$. These classes do not in fact depend on $V_\bullet^{ref}$, and we write $\sigma_w\coloneqq[X_{w_ow}(V_\bullet^{ref})]\in H^{2\ell(w)}(\flag{n})$.
The affine paving by Schubert cells implies that these \emph{Schubert classes} $\sigma_w$ form a linear basis of $H^*(\flag{n})$,
\begin{equation}
\label{eq:schubert_class_basis}
H^*(\flag{n})=\bigoplus_{w\in \sgrp_n}\bQ\sigma_w.
\end{equation}
Now given $Y$ irreducible of dimension $d$, we have an expansion of its fundamental class
\begin{equation}
\label{eq:class_expansion}
[Y]=\sum_wb_w\sigma_w,
\end{equation}
where the sum is over permutations of length $\ell(w_o)-d$. Then an important fact is that {\em $b_w$ is a nonnegative integer}. Indeed, $b_w$ can be interpreted as the number of points in the intersection of $Y$ with $X_{w}(V_\bullet^{ref})$ where $V_\bullet^{ref}$ is a \emph{generic} flag.\smallskip

One of the most important problems is to give a combinatorial interpretation to the coefficients when $Y=X_u(V^{std}_\bullet)\cap X_{w_ov}(V^{opp}_\bullet)$ with $u,v\in\sgrp_n$, that is $Y$ is a \emph{Richardson variety}.  Indeed the coefficients $b_w$ in this case are exactly the generalized LR coefficients $c_{uv}^w$ encoding the cup product in cohomology:
\begin{equation}
\label{eq:structure_coefficients}
\sigma_u\cup \sigma_v=\sum_{w\in\sgrp_n}c_{uv}^w \sigma_w.
\end{equation}

\subsection{Borel presentation and Schubert polynomials}

 Let $\bQ[\alpx_n]\coloneqq \bQ[x_1,\ldots, x_n]$ be the polynomial ring in $n$ variables.
 We denote the space of homogeneous polynomials of degree $k\geq 0$ in $\bQ[\alpx_n]$ by $\bQ^{(k)}[\alpx_n]$.
 Let $\Lambda_n\subseteq \bQ[\alpx_n]$ be the subring of symmetric polynomials in $x_1,\ldots,x_n$, and $I_n$ be the ideal of $\bQ[\alpx_n]$ generated by the elements $f\in \Lambda_n$ such that $f(0)=0$. Equivalently, $I_n$ is generated as an ideal by the elementary symmetric polynomials $e_1,\dots,e_n$. The quotient ring $R_n=\bQ[\alpx_n]/I_n$ is the \emph{coinvariant ring}.\smallskip

 Let $\divdiff_i$ be the divided difference operator on $\bQ[\alpx_n]$, given by
\begin{equation}
\divdiff_i(f)=\frac{f-s_i\cdot f}{x_i-x_{i+1}}.
\end{equation}

Define the \emph{Schubert polynomials} for $w\in \sgrp_n$ as follows: $\schub_{w_o}=x_1^{n-1}x_2^{n-2}\cdots x_{n-1}$, while if $i$ is a descent of $w$, let $\schub_{ws_i}=\divdiff_i\schub_w$. These are well defined since the $\divdiff_i$ satisfy the braid relations. For $w\in\sgrp_n$, the Schubert polynomial $\schub_w$ is a homogeneous polynomial of degree $\ell(w)$ in $\bQ[\alpx_n]$.
In fact Schubert polynomials are well defined for $w\in\sgrp_\infty$. Moreover, when $w\in\sgrp_\infty$ runs through all permutations whose largest descent is at most $n$, the  Schubert polynomials $\schub_w$  form a basis $\bQ[\alpx_n]$. \smallskip

Now consider the ring homomorphism
 \begin{equation}
 \label{eq:borel}
  j_n:\bQ[x_1,\ldots, x_n]\to H^*(\flag{n})
\end{equation} given by
  $j_n(x_i)=\sigma_{s_{i}}-\sigma_{s_i-1}$ for $i>1$ and $j_n(x_1)=\sigma_{s_{1}}$ (this is equivalent to the usual definition in terms of Chern classes). Then we have the following theorem, grouping famous results of Borel \cite{Bor53} and Lascoux and Sch\"utzenberger \cite{Las82}, see also \cite[Section 3.6]{Man01}.

 \begin{theorem}
 \label{thm:borel_LS}
 The map $j_n$ is surjective and its kernel is $I_n$. Therefore $H^*(\flag{n})$ is isomorphic as an algebra to $R_n$. Furthermore, $j_n(\schub_w)=\sigma_w$ if $w\in \sgrp_n$, and $j_n(\schub_w)=0$ if $w \in \sgrp_\infty-\sgrp_n$ has largest descent at most $n$.
 \end{theorem}

It follows immediately that the product of Schubert polynomials is given by the structure coefficients in \eqref{eq:structure_coefficients}: If $u,v\in\sgrp_n$, then
\begin{equation}
\label{eq:structure_coefficients_schubert}
\schub_u\schub_v=\sum_{w\in\sgrp_n}c_{uv}^w \schub_w \mod I_n.
\end{equation}
\noindent It is also possible to work directly in  $\bQ[\alpx_n]$ and not the quotient $R_n$: the coefficients $c_{uv}^w$ are well defined for $u,v,w\in\sgrp_\infty$, and one has

\begin{equation}
\label{eq:structure_coefficients_schubert_2}
\schub_u \schub_v=\sum_{w\in\sgrp_\infty}c_{uv}^w \schub_w.
\end{equation}

\subsection{Expansion in Schubert classes and degree polynomials}
\label{sub:expansion_degree}

Given $\beta\in H^*(\flag{n})$, let $\int\beta$ be the coefficient of $\sigma_{w_o}$ in the Schubert class expansion. Then we have the natural \emph{Poincar\'e duality} pairing on $H^*(\flag{n})$ given by $(\alpha,\beta)\mapsto \int(\alpha\cup\beta)$. The Schubert classes are known to satisfy $\int \sigma_u\cup \sigma_v=1$ if $u=w_ov$ and $0$ otherwise, so that the pairing is nondegenerate.
If $A,B\in \bQ[\alpx_n]$ are such that $j_n(A)=\alpha,j_n(B)=\beta$, then one can compute the pairing explicitly by:
\begin{equation}
\label{eq:pairing_explicit}
\int(\alpha\cup\beta)=\divdiff_{w_o}(AB)(0),
\end{equation}
where the right hand side denotes the constant term in $\divdiff_{w_o}(AB)$.\smallskip

The rest of this section is certainly well known to specialists, though  perhaps not presented in this form.
 We simply point out that given a cohomology class, computing its expansion in terms of Schubert classes and its degree polynomial correspond to evaluating a given linear form on two different families of polynomials.

 Let us fix $\alpha\in H^{n(n-1)-2p}(\flag{n})$. Our main interest is to consider $\alpha=[Y]$ where $Y$ is an irreducible closed subvariety of $\flag{n}$ of dimension $p$.
Associated to $\alpha$ is the linear form $\psi_\alpha:\beta\mapsto \int (\alpha\cup \beta)$ defined on $H^*(\flag{n})$. It vanishes if $\beta$ is homogeneous of degree $\neq 2p$, which leads to the following definition.

\begin{definition} Given $\alpha\in H^{n(n-1)-2p}(\flag{n})$ define the linear form $\phi_\alpha:\bQ^{(p)}[\alpx_n]\to \bQ$ by $\phi_\alpha(P)=\psi_\alpha(j_n(P))$ where $j_n$ is the Borel morphism defined earlier.
\end{definition}

Note that by definition, $\phi_\alpha$ vanishes on $\bQ^{(p)}[\alpx_n]\cap I_n$. For polynomials $A,P\in\bQ[\alpx_n]$ such that $j_n(A)=\alpha$, we have by \eqref{eq:pairing_explicit} the expression
\begin{equation}
 \label{eq:phi_alpha}
\phi_\alpha(P)=\divdiff_{w_o}(AP)(0).
\end{equation}
The coefficient $b_w$ in the expansion $\alpha=\sum_wb_w\sigma_w$ is given by
 \begin{equation}
 \label{eq:coefficient_extraction}
 b_w=\phi_\alpha(\schub_{w_ow})=\divdiff_{w_o}(\schub_{w_ow}A)(0).
 \end{equation}
  Indeed $j_n(\schub_{w_ow})=\sigma_{w_ow}$ by Theorem~\ref{thm:borel_LS}, and we use the duality of Schubert classes $\int\sigma_u\cup\sigma_v=0$ unless $v=w_ou$ where it is $1$.\smallskip

 The \emph{degree polynomial} of $\alpha$ is defined by \[\phi_\alpha((\lambda_1x_1+\cdots+\lambda_nx_n)^p),\] see \cite{HarHor18,PS09}. It is a polynomial in $\lambda=(\lambda_1,\ldots,\lambda_n)$, where coefficients are given by applying $\phi_\alpha$ to a monomial. When $\alpha=[Y]$ for a subvariety $Y$, and $\lambda\in \bQ^n$ is a strictly dominant weight $\lambda_1>\cdots>\lambda_n\geq 0$, $\phi_\alpha((\lambda_1x_1+\cdots+\lambda_nx_n)^p)$
 gives the degree of $Y$ in its embedding in $\mathbb{P}(V_\lambda)$  where $V_{\lambda}$ denotes the irreducible representation of $GL_n$ with highest weight $\lambda$.

The degree polynomials $\cD_w(\lambda_1,\ldots,\lambda_n)$ of Schubert classes $\sigma_w$ are studied in \cite{PS09}. Note that if $\alpha=\sum_{w}b_w\sigma_w$ as above, then by linearity the degree polynomial of $\alpha$ is $\sum_w b_w \cD_w(\lambda_1,\ldots,\lambda_n)$.

\subsection{Pipe dreams}

The \emph{BJS formula} of Billey, Jockusch and Stanley \cite{Bil93} is an explicit nonnegative expansion of $\schub_w$  in the monomial basis:
\begin{equation}
\label{eq:bjs}
\schub_w(x_1,\ldots,x_n)=\sum_{\mathrm{i}\in \reduced(w)} \sum_{b\in C(\mathrm{i})} \alpx^b,
\end{equation}
where $C(\mathrm{i})$ is the set of compositions $b_1\leq \ldots\leq b_{l}$ such that $1\leq b_j\leq i_j$, and $b_j<b_{j+1}$ whenever $i_j<i_{j+1}$. Additionally, $\alpx^b$ is the monomial $x_1^{b_1}\cdots x_{l}^{b_l}$. \smallskip

The expansion in \eqref{eq:bjs} has a nice combinatorial version with {\em pipe dreams} (also known as rc-graphs), which we now describe. Let $\mathbb Z_{> 0}\times \mathbb Z_{> 0}$ be the semi-infinite grid, starting from the northwest corner.  Let $(i,j)$ indicate the position at the $i$th row from the top and the $j$th column from the left.  A {\em pipe dream} is a tiling of this grid with $+$'s (pluses) and \includegraphics[scale=.6]{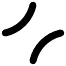}'s (elbows) with a finite number of $+$'s. The \emph{size} $|\gamma|$ of a pipe dream $\gamma$ is the number of $+$'s.\smallskip

Any pipe dream can be viewed as composed of \emph{strands}, which cross at the $+$'s.
Strands naturally connect bijectively rows on  the left edge of the grid and columns along the top; let $w_{\gamma}(i)=j$ if the $i$th row is connected to the $j$th column, which defines a permutation $w_\gamma\in\sgrp_\infty$.\smallskip

Say that $\gamma$ is {\em reduced} if $|\gamma|=\ell(w_\gamma)$; equivalently, any two strands of $\gamma$ cross at most once. We let $\pipe(w)$ be the number of reduced pipe dreams $\gamma$ such that $w_\gamma=w$.
Notice that if $w\in \sgrp_n$ then the $+$'s in any $\gamma\in \pipe(w)$ can only occur in positions $(i,j)$ with $i+j<n$, so we can restrict the grid to such positions.

\begin{figure}[!ht]
\includegraphics[scale=0.6]{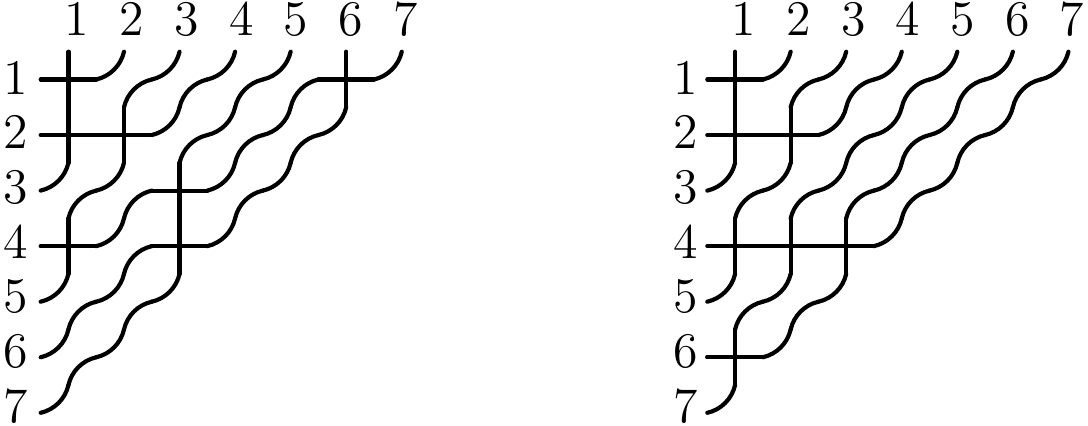}
\caption{\label{fig:pipe_dreams}  Two reduced pipe dreams with permutation $w_\gamma=2417365$. On the right is the bottom pipe dream attached to this permutation.}
 \end{figure}

Given $\gamma\in \pipe(w)$, define $c(\gamma)\coloneqq (c_1,c_2,\ldots)$ where $c_i$ is the number of $+$'s on the $i$th row of $\gamma$. Then the BJS expansion~\eqref{eq:bjs} can be rewritten as follows \cite{Bil93,Man01}:
 \begin{equation}
\label{eq:pipe_dreams}
{\schub_{w}}=\sum_{\gamma\in \pipe(w)} \alpx^{c(\gamma)}.
\end{equation}
 \begin{wrapfigure}{r}{0.3\textwidth}
\begin{center}
\includegraphics[scale=0.5]{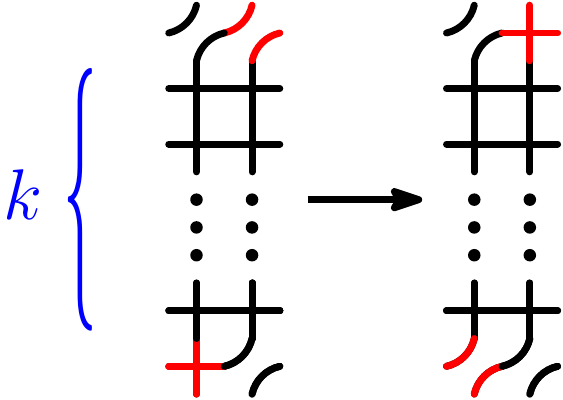}
\end{center}
\end{wrapfigure}
Given $w\in \sgrp_\infty$, let $(c_1,c_2,\ldots)=\code{w}$. The {\em bottom pipe dream} $\gamma_w\in \pipe(w)$ consists of $+'s$ in columns $1,\ldots,c_i$ for each row $i=1,\ldots,n$; note that $c(\gamma_w)=\code{w}$.

A {\em ladder move} is a local operation on pipe dreams illustrated on the right: here $k$ can be any nonnegative integer. When $k=0$ this is called a {\em simple} ladder move.
The following result shows how to easily generate all pipe dreams attached to a given permutation.

\begin{theorem}(\cite[Theorem 3.7]{Ber93})
\label{thm:ladder_moves}
Let $w\in \sgrp_n$. If $\gamma\in \pipe(w)$, then $\gamma$ can be obtained by a sequence of ladder moves from $ \gamma_w$.
\end{theorem}

\begin{definition}
For any $w\in\sgrp_\infty$, define the \emph{principal specialization} $\nu_w$ of the Schubert polynomials $\schub_w$ by  $\nu_w=\schub_w(1,1,\ldots)$.
\end{definition}

By the expansion \eqref{eq:pipe_dreams}, one has the combinatorial interpretation
\begin{equation}
\label{eq:nu_equals_PD}
\nu_w=|\pipe(w)|.
\end{equation}
 An alternative expression for $\nu_w$ is given by \emph{Macdonald's reduced word identity} \cite{Macdonald}
\begin{equation}
\label{eq:macdonald}
\nu_{w}=\frac{1}{\ell(w)!}\sum_{\mathbf{i}\in \reduced(w)}i_1i_2\cdots i_{\ell(w)}.
\end{equation}

A deeper study of Macdonald's reduced word identity and its generalizations has seen renewed interest recently and has brought forth various interesting aspects of the interplay between Schubert polynomials, combinatorics of reduced words, and differential operators on polynomials.
We refer the reader to \cite{Bil19,Ham20,Wei18,MPPa19} for more details.
As we shall see in the next section, an expression rather reminiscent of the right hand side of \eqref{eq:macdonald} plays a key role in our quest to obtain the Schubert expansion for  $\tau_n=[\Perm_n]$, and its appearance in this context begs for deeper explanation.

\section[Formulas]{Formulas for \texorpdfstring{$a_w$}{}}
\label{sec:formulas}

Recall that we want to investigate the numbers $a_w$ occurring in the Schubert class expansion
\begin{equation*}
 \class_n=\sum_{w\in \sgrpp_n}a_w\sigma_{w_ow}\in H^{*}(\flag{n}).
\end{equation*}

Now $\tau_n$ is the class of the variety $\Perm_n$, so by the classical results recalled in Section~\ref{sub:cohomology}, we know that the $a_w$ are nonnegative integers: namely $a_w$ is the number of points in the intersection of $\Perm_n$ with a Schubert variety $X_{w_ow}(V_\bullet)$ where $V_\bullet$ is a \emph{generic} flag.

In this section we use two approaches \textemdash{} the first due to Klyachko \cite{Kly85,Kly95}, the second due to Anderson-Tymoczko~\cite{And10}\textemdash{} to arrive at algebraic expressions for the numbers $a_w$. These are given in Theorems \ref{thm:klyachko} and \ref{thm:anderson_tymoczko} respectively, and both expressions will be exploited to extract various properties of the numbers $a_w$.

\subsection{\texorpdfstring{$a_w$}{} via Klyachko's approach}
\label{sub:Klyachko}

We will extract our first expression from the results of~\cite{Kly85,Kly95}. 
Note that \cite{Kly85} is a two page summary of results (in English), while \cite{Kly95} proves these results and expands on them, and is written in Russian. We describe the two theorems of significance for us in Section~\ref{sec:klyachko_formula}, giving a slightly simplified proof for the second one.

Given $w\in \sgrp_\infty$ of length $\ell=\ell(w)$, consider the polynomial in $\bQ[x_1,x_2,\ldots]$:
\begin{align}
\label{eq:M}
\M_w(x_1,x_2,\ldots)&\coloneqq\sum_{\mathbf{i}={i_1}{i_2}\cdots {i_{\ell}}\in \reduced(w)}x_{i_1}x_{i_2}\cdots x_{i_{\ell}}=\sum_{\mathbf{i}\in \reduced(w)}\alpx^{c(\mathbf{i})},
\end{align}
where $c(\mathbf{i})=(c_1,c_2,\ldots)$ and $c_j$ is the number of occurrences of $j$ in $\mathbf{i}$. If $w\in \sgrp_n$, then $\M_w$ is a polynomial in $x_1,\ldots,x_{n-1}$.
Notice that Macdonald's formula \eqref{eq:macdonald} states that \[M_w(1,2,\ldots)=\ell!\cdot \nu_w.\]

For $n\geq 3$, let $\D_n$ be the commutative $\bQ$-algebra with generators $u_1,\ldots,u_{n-1}$ and defining relations
\[\begin{cases}
2u_i^2=u_iu_{i-1}+u_iu_{i+1}\text{ for }1<i<n-1;\\
2u_1^2=u_1u_{2};\\
2u_{n-1}^2=u_{n-1}u_{n-2}.
\end{cases}\]Given $I=\{i_1<\cdots<i_{j}\}\subset [n-1]$, define $u_I\coloneqq u_{i_1}\cdots u_{i_j}$.
Then the elements $u_I$, $I\subset [n-1]$ form a basis of $\D_n$. Given $U=\sum_Ic_Iu_I\in \D_n$, let $\int_{\D_n} U$ be the top coefficient $c_{[n-1]}$.

\begin{theorem}
\label{thm:klyachko}
For any $w\in\sgrpp_n$, we have
\begin{align*}
	a_w=\int_{\D_n} M_w(u_1,u_2,\ldots,u_{n-1}).
\end{align*}\end{theorem}

\begin{proof}
	This is a light reformulation of Klyachko's work~\cite{Kly85,Kly95}, specialized to type A. 
	The rational cohomology ring of $\Perm_n$ is computed in this work. $\sgrp_n$ acts on this ring, and the corresponding subring of invariants is shown to be isomorphic to the algebra $\D_n$ above. In this presentation, the fundamental class of $\Perm_n$ is represented by $u_{[n-1]}/(n-1)!$.

	Now the embedding $\Perm_n\to \flag{n}$ gives a pullback morphism $H^*(\flag{n})\to \D_n$, under which the image of the Schubert class $\sigma_w$ is $M_w(u_1,u_2,\ldots,u_{n-1})/\ell(w)!$.	Let $w\in\sgrpp_n$.
	We have $a_w=\int \sigma_w\cup\class_n= \int \sigma_w\cup [\Perm_n]$. 	By pulling back the computation to $\D_n$, we get the result.
\end{proof}

\subsection{\texorpdfstring{$a_w$}{} via Anderson--Tymoczko's approach}
\label{sub:AT}

We have already encountered the operator of divided symmetrization $\ds{\cdot}{n}$ in the introduction.

\begin{theorem}
\label{thm:anderson_tymoczko}
For any $w\in\sgrpp_n$,
\begin{equation}
\label{eq:aw_as_ds}
a_w=\ds{\schub_w(x_1,\ldots,x_n)}{n}.
\end{equation}
\end{theorem}

We recall some relevant results from \cite{And10}. A \emph{Hessenberg function} $h:[n]\to [n]$ is a function satisfying the condition that $i\leq h(i)$ for all $i\in [n]$ and $h(i)\leq h(j)$ for all $1\leq i<j\leq n$.
Given an $n \times n$ matrix $X$ and a Hessenberg function $h :[n] \to [n]$, the \emph{Hessenberg
variety} (in type A) associated with $X$ and $h$ is defined to be
 \[
\Hess(X, h) \coloneqq \{{V_{\bullet} \in \flag{n}\suchthat X\cdot V_j \subset V_{h(j)} \text{ for all } j \in [n]}\}.
\]

We consider $\Hess(X,h)$ for $X$ a \emph{regular} matrix: this means that $X$ has exactly one Jordan block attached to each eigenvalue. Since regular Hessenberg varieties form a flat family~\cite{AFZ20_bis} the class $\Sigma_h=[\Hess(X,h)]\in H^*(\flag{n})$ does not depend on $X$. \smallskip

By relating $\Hess(X,h)$ to a degeneracy locus when $X$ is regular semisimple, Anderson and Tymoczko~\cite{And10} express $\Sigma_h$ as a certain specialization of a \emph{double Schubert polynomial} \cite{Man01}. 
We identify $H^*(\flag{n})$ and $R_n=\bQ[\alpx_n]/I_n$ thanks to Theorem \ref{thm:borel_LS}. 
The main result of \cite{And10} is
\begin{align}
\label{eq:regular_class_doubleschubert}
\Sigma_h&=\schub_{w_h}(x_1,\cdots,x_n;x_n,\cdots,x_1)\mod I_n
\\&=\prod_{\stackrel{1\leq i<j\leq n}{j>h(i)}}(x_i-x_j)\mod I_n.\label{eq:regular_class_product}
\end{align}
where $w_h$ is the permutation given by $\code{w_h^{-1}}=(n-h(1),\ldots,n-h(n))$. The simple product form in \eqref{eq:regular_class_product} comes from the fact that $w_h$ is a dominant permutation, cf.~\cite[Proposition 2.6.7]{Man01}.
\medskip

Now in the case of $h=(2,3,\ldots,n,n)$, we have that $\Sigma_h=\tau_n$ by definition and thus
 \[\class_n=\prod_{\stackrel{1\leq i<j\leq n}{j>i+1}}(x_i-x_j)\mod I_n.\]
 Following the terminology of Section \ref{sub:expansion_degree}, consider the linear form $\phi_{\class_n}$ defined on $\bQ^{(n-1)}[\alpx_n]$ by
\[\phi_{\class_n}(P)=\divdiff_{w_o}(P\prod_{\stackrel{1\leq i<j\leq n}{j>i+1}}(x_i-x_j))\]

We know that $\phi_{\class_n}(\schub_w)=a_w$ by \eqref{eq:coefficient_extraction}, so that Theorem~\ref{thm:anderson_tymoczko} follows immediately from the next proposition.

\begin{proposition}
\label{prop:phi_tau_n}
For any  $P\in \bQ^{(n-1)}[\alpx_n]$,\[\phi_{\class_n}(P)=\ds{P}{n}.\]
\end{proposition}

\begin{proof}
	Let $\anti_n$ and $\sym_n$ denote the antisymmetrizing operator $\sum_{\sigma\in\sgrp_n} \epsilon(\sigma)\sigma$ and symmetrizing operator $\sum_{\sigma\in\sgrp_n}\sigma$ acting on $\bQ[\alpx_n]$ respectively. Here the action of the symmetric group permutes indeterminates, and $\epsilon(\sigma)$ denotes the \emph{sign} of $\sigma$.
	Let $\Delta_n$ denote the usual Vandermonde determinant given by $\prod_{1\leq i<j\leq n}(x_i-x_j)$.

	One has $\divdiff_{w_o}=\frac{1}{\Delta_n}\anti_n$ \cite[Proposition 2.3.2]{Man01} so that

	\begin{align*}
		\phi_{\class_n}(P)&=\frac{1}{\Delta_n}\anti_n\left(P\prod_{1\leq i<j\leq n,j\neq i+1}(x_i-x_j)\right)=\frac{1}{\Delta_n}\anti_n\left(\frac{P\Delta_n}{\prod_{1\leq i\leq n-1}(x_i-x_{i+1})}\right)\\
		&=\frac{\Delta_n}{\Delta_n}\sym_n\left(\frac{P}{\prod_{1\leq i\leq n-1}(x_i-x_{i+1})}\right)=\ds{P}{n}.
	\end{align*}
	Here we used the fact that $\sigma(\Delta_n)=\epsilon(\sigma)\Delta_n$ between the first and second lines.
\end{proof}

\begin{remark}
\label{rem:volume_and_degree}
	There is an alternative way to prove Proposition~\ref{prop:phi_tau_n} (equivalently, Theorem~\ref{thm:anderson_tymoczko}), which illuminates why the operator of divided symmetrization occurs in our context.

	It is well known that $\Perm_n$ is a smooth toric variety. 
	Therefore its degree in the embedding $\mathbb{P}(V_\lambda)$ for $\lambda$ strictly dominant is given by the (normalized) volume of its associated polytope. This polytope is the permutahedron with vertices given by permutations of $(\lambda_1,\ldots,\lambda_n)$; see next section for more details. The volume was computed by Postnikov \cite[Theorem 3.2]{Pos09} as a polynomial in $(\lambda_1,\ldots,\lambda_n)$; his result is that the degree polynomial of $\class_n=[\Perm_n]$ is $\ds{(\lambda_1x_1+\cdots+\lambda_nx_n)^{n-1}}{n}$. Since this degree polynomial completely characterizes $\phi_{\tau_n}$, this proves Proposition~\ref{prop:phi_tau_n}.
\end{remark}

\section[Mixed Eulerian numbers]{Mixed Eulerian numbers}
\label{sec:mixed_eulerian}
We turn our attention to an intriguing family of positive integers introduced by Postnikov \cite{Pos09}.
These are the \emph{mixed Eulerian numbers} $A_{c_1,\dots,c_n}$ indexed by weak compositions $c\coloneqq (c_1,\dots,c_n)$ where $\sum_{1\leq i\leq n}c_i=n-1$. We denote by $\wcp{n}$ the set of such compositions.
Recall that a \emph{weak composition} $(c_1,\dots,c_n)$ is simply a sequence of nonnegative integers. A \emph{strong composition} $a=(a_1,\dots,a_p)$ is composed of positive integers, and we write $a\vDash N$ if $\sum_{1\leq i\leq p}a_i=N$. If $c=(0^{k-1},n-1,0^{n-k})$ for some $1\leq k\leq n$,  then $A_c$ equals the classical Eulerian number enumerating permutations in $S_{n-1}$ with $k-1$ descents, which explains the name for the $A_c$ in general.

We collect here various aspects of the mixed Eulerian numbers that shall play a key role in what follows, beginning by explaining how they arise in Postnikov's work.

Given $\lambda\coloneqq (\lambda_1\geq \dots\geq \lambda_n)\in \bR^n$, let $P_{\lambda}$ be the \emph{permutahedron} in $\bR^n$ obtained by considering the convex hull of all points in the $\sgrp_n$-orbit of $\lambda$.
 Let $\mathrm{Vol}(P_{\lambda})$ denote the usual $(n-1)$-dimensional volume of the polytope obtained by projecting $P_{\lambda}$ onto the hyperplane defined by the $n$-th coordinate equaling $0$.

\medskip

 By \cite[Theorem 3.1]{Pos09}, we have that
 \begin{align}\label{eqn:volume_ds_formula}
 	(n-1)!\mathrm{Vol}(P_{\lambda})=\ds{(\lambda_1x_1+\cdots + \lambda_nx_n)^{n-1}}{n}
 \end{align}
 Setting $u_i=\lambda_i-\lambda_{i+1}$ for $1\leq i\leq n-1$, and $u_n=\lambda_n$, we have that
 \begin{align}
 	\sum_{1\leq i\leq n}\lambda_ix_i=\sum_{1\leq i\leq n}u_i(x_1+\dots+x_i).
 \end{align}
 For brevity, set $y_i$ equal to $x_1+\dots + x_i$, and for $c=(c_1,\dots,c_n)$ define
 \begin{align}\label{eqn:def_yc}
 y^c\coloneqq \prod_{1\leq i\leq n}y_i^{c_i}.
 \end{align}
 This given, rewrite \eqref{eqn:volume_ds_formula} to obtain
 \begin{align}
 	\mathrm{Vol}(P_{\lambda})=\sum_{c\in \wcp{n}}\ds{y^c}{n}\frac{u_1^{c_1}\dots u_n^{c_n}}{c_1!\cdots c_n!}.
 \end{align}
We define the \emph{mixed Eulerian number} $A_c$ to be $\ds{y^c}{n}$, and note that Postnikov \cite[Section 16]{Pos09} interprets them as certain mixed volumes up to a normalizing factor, see below.

 Observe that $\ds{y^c}{n}$ is equal to $0$ if $c_n>0$ because of the presence of the symmetric factor $(x_1+\dots+x_n)^{c_n}$ \cite[Corollary 3.2]{DS}.
 Hence one may safely restrict one's attention to mixed Eulerian numbers $A_{c_1,\dots,c_n}$ where $c_n=0$.\footnote{The reader comparing our notation to that in \cite{Pos09} should note that Postnikov works under the tacit assumption that $c_n=0$.}
Henceforth, if we index a mixed Eulerian number by an $(n-1)$-tuple summing to $n-1$, we are implicitly assuming that $c_n=0$.

 The key fact about the mixed Eulerian numbers $A_{(c_1,\dots,c_{n-1})}$ pertinent to our purposes is that they are positive integers.
 As explained in \cite[Section 16]{Pos09}, $A_{(c_1,\dots,c_{n-1})}$ equals the mixed volume of the Minkowski sum of hypersimplices $c_1\Delta_{1,n}+\cdots +c_{n-1}\Delta_{n-1,n}$ times $(n-1)!$, which implies positivity.
By performing a careful analysis of the volume polynomial $\mathrm{Vol}(P_{\lambda})$, Postnikov further provides a combinatorial interpretation for the $A_{(c_1,\dots,c_{n-1})}$ in terms of weighted binary trees; see \cite[Theorem 17.7]{Pos09}.
A more straightforward combinatorial interpretation for these numbers was provided by Liu \cite{Liu16}, in terms of certain permutations with a recursive definition. We omit further details and refer the reader to the articles. Instead we move on to describe some beautiful results due to Petrov \cite{Pet18}. Interestingly, Petrov does not mention mixed Eulerian numbers in his statements, which we believe deserve to be more widely known in this context.

We begin by listing some relations satisfied by the mixed Eulerian numbers that characterize them uniquely. The reader should pay particular heed to the third relation below and compare it with the presentation of Klyachko's algebra $\D_n$ from before.

\begin{lemma}[\cite{Pet18}]\label{lem:uniqueness_via_recurrence_petrov}
  For a fixed positive integer $n$, the mixed Eulerian numbers $A_{(c_1,\dots,c_n)}$ are completely determined by the following relations:
  \begin{enumerate}
      \item $A_{(c_1,\dots,c_n)}=0$ if $c_n>0$.
      \item $A_{(1^{n-1},0)}=(n-1)!$.
      \item $2A_{(c_1,\dots,c_n)}= A_{(c_1,\dots,c_{i-1}+1,{c_i}-1,\dots,c_n)}+A_{(c_1,\dots,c_i-1,c_{i+1}+1,\dots,c_n)}$ if $i\leq n-1$ and $c_i\geq 2$.
  \end{enumerate}
  In the last relation, we interpret $c_{0}$ to be $c_n$.
\end{lemma}
\begin{proof}
 Let us sketch Petrov's proof. We have already addressed the first point. The second relation follows immediately by realizing that $y^c$ is a sum of $(n-1)!$  monomials when $c=(1^{n-1},0)$, and each such monomial contributes $1$ upon divided symmetrization; see \cite[Section 3.3]{DS}.
For the third relation we refer the reader to \cite[Theorem 4]{Pet18}; it relies on a nice  property of divided symmetrization.

The uniqueness follows from the maximum principle: given two solutions to these relations, conside their difference $\delta_{(c_1,\dots,c_n)}$. Assume that $\delta$ achieves its maximum value $m$ at $(c_1,\dots,c_n)$: then the third relation implies that $m$ is also achieved at $(c_1,\dots,c_{i-1}+1,{c_i}-1,\dots,c_n)$ and $(c_1,\dots,c_i-1,c_{i+1}+1,\dots,c_n)$. Applying this argument repeatedly, we can reach all compositions  $c_{\mathrm{terminal}}$ that have all but one part equal to $1$. Since $\delta_{c_{\mathrm{terminal}}}=0$ by the first two relations, this shows that $\delta=0$ everywhere.
\end{proof}

{\noindent \bf Probabilistic interpretation.} Petrov turns this characterization into a probabilistic process as follows: Consider $n-1$ coins distributed among the vertices of a regular $n$-gon, denoted by $v_1$ through $v_n$ going cyclically.
A \emph{robbing} move consists of picking a vertex $v_i$ that has at least $2$ coins, and transferring one coin to either vertex $v_{i-1}$ or $v_{i+1}$ with equal probability. Proceed making such moves until no vertices can be robbed any further. The process terminates almost surely. Note that there are $n$ terminal configurations, each having $1$ coin at $n-1$ sites and $0$ on the remaining site.
Given $c_1,\dots,c_n$ such that $\sum_{1\leq i\leq n}c_i=n-1$, let $\mathrm{prob}(c_1,\dots,c_n)$ denote the probability of starting from the initial assignment of $c_i$ coins to $v_i$ and ending in the configuration where $v_n$ has no coins.
\begin{theorem}\textup{(}\cite[Theorem 5]{Pet18}\textup{)}
  Assuming the notation established earlier, we have that
  \[
    \mathrm{prob}(c_1,\dots,c_n)=\frac{A_{(c_1,\dots,c_n)}}{(n-1)!}.
  \]
\end{theorem}
Petrov arrives at this result by noting that $(n-1)!\mathrm{prob}(c_1,\dots,c_n)$ satisfies the defining relations of the mixed Eulerian numbers listed in Lemma~\ref{lem:uniqueness_via_recurrence_petrov}.

\begin{example}
Suppose $(c_1,c_2,c_3,c_4)=(2,1,0,0)$.
It can be checked that $p\coloneqq \mathrm{prob}(c_1,\dots,c_4)$ satisfies $p=\frac{1}{4}(1+p)$ implying that $p=\frac{1}{3}$.
This in turn implies that $A_{(2,1,0,0)}=2$, which is verified easily by expanding $y^{(2,1,0,0)}=x_1^3+x_1^2x_2$ and noting that both monomials give $1$ upon divided symmetrization.
\end{example}

The preceding probabilistic interpretation renders transparent an interesting  relation satisfied by the mixed Eulerian numbers.
Define the \emph{cyclic class} of a sequence $c\coloneqq (c_1,\dots,c_n)\in \wcp{n}$ to be the set of all sequences obtained as cyclic rotations  of $c$.
Let us denote this cyclic class by $\mathsf{Cyc}(c)$. It is clear that $|\mathsf{Cyc}(c)|=n$.

\begin{proposition}\textup{(}\cite[Theorem 16.4]{Pos09}, \cite[Theorem 4]{Pet18}\textup{)}
\label{prop:mixed_cyclic}
For $c\in \wcp{n}$, we have that
\[
\sum_{c'\in \mathsf{Cyc}(c)}\frac{A_{c'}}{(n-1)!}=1.
\]
\end{proposition}

We conclude this section with a discussion on a special class of sequences $c$.
We say that  $c\in \wcp{n}$ is \emph{connected} if $c$ comprises a solitary contiguous block of positive integers and has $0$s elsewhere.
For instance $(0,1,1,2,0)$ is connected, whereas $(0,1,0,3,0)$ is not.
Our next result is presented in recent work of Berget, Spink and Tseng \cite[Section 7]{Ber20}, and was also established independently by the authors.
\begin{proposition}\label{prop:mixed_connected}
	Let ${\bf a}=(a_1,\ldots,a_p)$ be a strong composition of $n-1$. For $i,j$ nonnegative integers let $0^i{\bf a}0^j$ denote the sequence obtained by appending $i$ 0s before ${\bf a}$ and $j$  0s after it.
	Consider the polynomial
	\[
	\widetilde{A}_{\bf a}(t)=\sum_{m=0}^{n-p-1}A_{0^m{\bf a}0^{n-p-m}} t^m.
	\]
	We have that
	\[
	\sum_{j\geq 0}(1+j)^{a_1}(2+j)^{a_2}\cdots (p+j)^{a_p}t^j =\frac{\widetilde{A}_{{\bf a}}(t)}{(1-t)^{n}} .
	\]
\end{proposition}

\begin{example}
	Consider $c=(3,0,0,0)\in \wcp{4}$. 
	Since $\sum_{j\geq 0}(j+1)^3t^j=\frac{1+4t+t^2}{(1-t)^4}$, Proposition~\ref{prop:mixed_connected} tells us that $A_{(3,0,0,0)}=1$, $A_{(0,3,0,0)}=4$ and $A_{(0,0,3,0)}=1$, which are the well-known Eulerian numbers counting permutations in $S_3$ according to descents.
\end{example}

\section[Properties]{Properties of the numbers \texorpdfstring{$a_w$}{}}
\label{sec:properties}

Our starting point in this section is Klyachko's Theorem~\ref{thm:klyachko}, from which we deduce a formula for $a_w$ in terms of mixed Eulerian numbers (Theorem~\ref{thm:positive_formula}). From the properties of these mixed Eulerian numbers reviewed in Propositions~\ref{prop:mixed_cyclic} and ~\ref{prop:mixed_connected}, we obtain related properties of $a_w$ in Theorems~\ref{thm:cyclic_sum} and ~\ref{thm:quasiindecomposable} respectively.

\subsection[Positive formula]{A positive formula for $a_w$ and first properties}
\label{sub:positivity}

To start with, we have the following invariance properties of $a_w$ easily deduced from Theorem~\ref{thm:klyachko}:

\begin{proposition}
\label{prop:stability}
For any $w\in\sgrpp_n$, $a_w=a_{w^{-1}}$ and $a_w=a_{w_oww_o}$.
\end{proposition}

\begin{proof}
We have the equality of polynomials $M_w=M_{w^{-1}}$ since $i_1\ldots i_{n-1}\mapsto i_{n-1}\ldots i_1$ is a bijection from $\reduced(w)$ to $\reduced(w^{-1})$, and so we can conclude by Theorem~\ref{thm:klyachko}.

Also, $i_1\cdots i_{n-1}\mapsto (n-i_{1})\cdots (n-i_{n-1})$ is a bijection from $\reduced(w)$ to $\reduced(w_oww_o)$, so $M_{w_oww_o}$ is obtained from $M_w$ after the substitution $x_i\mapsto x_{n-i}$. Because of the symmetry in the presentation of $\D_n$, Theorem~\ref{thm:klyachko} gives us again that $a_w=a_{w_oww_o}$.
\end{proof}

The invariance under $w_o$-conjugation is also a special case of~\cite[Proposition 3.8]{And10}, which can be explained geometrically via the duality on $\flag{n}$.
The authors know of no such explanation for the invariance under taking inverses.

We can now state our first formula.

\begin{theorem}
\label{thm:positive_formula}
For any $w\in\sgrpp_n$ and $\mathbf{i}\in \reduced(w)$, let $c(\mathbf{i})=(c_1,\ldots,c_{n-1})$ where $c_j$ counts the occurrences of
$j$ in $\mathbf{i}$. Then
\begin{equation}
\label{eq:positive_formula}
a_w=\sum_{\mathbf{i}\in \reduced(w)}\frac{A_{c(\mathbf{i})}}{(n-1)!}.
\end{equation}
\end{theorem}

\begin{proof}
By Theorem~\ref{thm:klyachko}, it is enough to show that, for any weak composition $c=(c_1,\dots,c_{n-1})$ of $n-1$,
\begin{equation}
\label{eq:uc_and_yc}
\int_{\D_n} u^c=\frac{A_{c}}{(n-1)!}.
\end{equation}

We now claim that $(n-1)!\int_{\D_n} u^c$ satisfies the three conditions of Lemma~\ref{lem:uniqueness_via_recurrence_petrov}. Indeed the first two are immediate, while the third follows precisely from the relations of $\D_n$. By uniqueness in Lemma~\ref{lem:uniqueness_via_recurrence_petrov}, $(n-1)!\int_{\D_n} u^c=A_{c}$ as wanted.

Equation~\eqref{eq:uc_and_yc} can also be deduced geometrically  from the interpretation of $A_c$ as a normalized mixed volume, cf. \cite{Ber20,Pos09}.
\end{proof}

\begin{example}
Consider $w=32415\in\sgrpp_5$. It has three reduced words $2123,1213$ and $1231$. Given that $A_{2,1,1,0}=6$ and $A_{1,2,1,0}=12$, we obtain $a_w=\frac{1}{24}(12+6+6)=1$.
\end{example}

The following immediate corollary answers a question asked in~\cite[Problem 6.6]{HarHor18}.

\begin{corollary}
\label{cor:aw_positivity}
For any $w\in\sgrpp_n$, $a_w>0$;
\end{corollary}

\begin{proof}
It follows directly from ~\eqref{eq:positive_formula} since it expresses $a_w$ as a nonempty sum of positive rational numbers.
\end{proof}

From Section \ref{sec:mixed_eulerian} we know also that $A_{c}\leq (n-1)!$ for any $c$, so that $a_w\leq |\reduced(w)|$ by Theorem~\ref{thm:positive_formula}. We will get a quantitative version of the inequality in Theorem~\ref{thm:cyclic_sum}.

\begin{remark}
It is worth remarking that if we consider the computation of $A_{(c_1,\dots,c_{n-1})}$ using its original definition, we must deal with $\ds{y_1^{c_1}\dots y_{n-1}^{c_{n-1}}}{n}$.
By using Monk's rule \cite{Mo59} repeatedly, we can express $y_1^{c_1}\dots y_{n-1}^{c_{n-1}}$ as a positive integral sum of certain Schubert polynomials in the variables $x_1,\dots,x_{n-1}$.
Applying divided symmetrization to the resulting equality results in an expression for $A_{(c_1,\dots,c_{n-1})}$ expressed as a positive integral combination of certain $a_w$'s.
It appears nontrivial to `invert' this procedure and obtain the expression in Theorem~\ref{thm:positive_formula} for the $a_w$.
At any rate, assuming the aforementioned theorem, one does obtain a curious expression for $A_{(c_1,\dots,c_{n-1})}$ in terms of other mixed Eulerian numbers with weights coming from certain chains in the Bruhat order. We omit the details.
\end{remark}

Let us also mention that the results of this section have analogues in other types, see Section~\ref{sec:further_remarks}.

\subsection{Indecomposable permutations and sum rules}
\label{sub:indecomposable}

In this section we establish two summatory properties of the numbers $a_w$, based on the notion of factorization of a permutation into indecomposables, which we now recall.\smallskip

Let $w_1,w_2 \in \sgrp_m\times \sgrp_p$ with $m,p>0$. The concatenation $w=w_1\times w_2\in S_{m+p}$ is defined by $w(i)=w_1(i)$ for $1\leq i\leq m$ and $w(m+i)=m+w_2(i)$ for $1\leq i\leq p$. This is an associative operation, sometimes denoted by $\oplus$ and referred to as connected sum.
A permutation $w\in \sgrp_n$ is called \emph{indecomposable} if it cannot be written as $w=w_1\times w_2$ for any $w_1,w_2 \in \sgrp_m\times \sgrp_p$ with $n=m+p$.
Note that the unique permutation of $1\in\sgrp_1$ is indecomposable.
The indecomposable permutations for $n\leq 3$ are  $1,21,231,312,321$, and their counting sequence is A003319 in \cite{oeis}.
Permutations can be clearly uniquely factorized  into indecomposables: given $w$ in $\sgrp_n$, it has a unique factorization
\begin{equation}
\label{eq:indecomposables}
w=w_1\times w_2\times\cdots \times w_k,
\end{equation}
where each $w_i$ is an indecomposable permutation in $\sgrp_{m_i}$ for certain $m_i>0$.
For instance $w=53124768\in \sgrp_8$ is uniquely factorized as $w=53124 \times 21 \times 1$.
We say that $w$ is {\em quasiindecomposable} if exactly one $w_i$ is different from $1$. Thus a quasiindecomposable permutation has the form $1^{i}\times u\times1^{j}$ for $u$ indecomposable $\neq 1$ and integers $i,j\geq 0$.

Given $w\in \sgrp_n$ decomposed as \eqref{eq:indecomposables}, its cyclic shifts $w^{(1)},\ldots,w^{(k)}$ are given by
\begin{equation}
\label{eq:cyclic_shifts}
w^{(i)}=(w_i\times w_{i+1} \cdots \times w_k) \times (w_1\times \cdots \times w_{i-1}).
\end{equation}
The cyclic shifts of $w=53124768$, decomposed above, are $w^{(1)}=w=53124768$, $w^{(2)}=21386457$ and $w^{(3)}=16423587$.

These notions are very natural in terms of reduced words: Let the \emph{support} of $w\in \sgrp_n$ be the set of letters in $[n-1]$ that occur in any reduced word for $w$. Then $w$ is indecomposable if and only if it has full support $[n-1]$. It  is quasiindecomposable if its support is an interval in $\mathbb{Z}_{>0}$.
Finally, the number $k$ of cyclic shifts of $w$ is equal to $n$ minus the cardinality of the support of $w$.

\begin{theorem}[Cyclic Sum Rule]
\label{thm:cyclic_sum}
Let $w\in \sgrpp_n$, and consider its cyclic shifts $w^{(1)},\ldots,w^{(k)}$ defined by~\eqref{eq:indecomposables} and~\eqref{eq:cyclic_shifts}. We have that
\begin{equation}
\label{eq:cyclic_sum}
\sum_{i=1}^k a_{w^{(i)}}= |\reduced(w)|.
\end{equation}
\end{theorem}

\begin{proof}
Let $\mathbf{i}=i_1\cdots i_{n-1}$ be a reduced word for $w=w^{(1)}$. Consider the words $\mathbf{i}[t]=(i_1+t)\cdots (i_{n-1}+t)$ for $t=0,\ldots,n-1$, where the values $i_j+t$ are considered as their residues modulo $n$ with representatives belonging to the interval $\{1,\ldots,n\}$. Let $0=t_1<\cdots<t_{k}$ be the values of $t$ for which $n$ does \emph{not} occur in $\mathbf{i}[t]$. Then in the notation of \eqref{eq:cyclic_shifts}, we have $t_j=\sum_{i=1}^{j-1} m_i$. Moreover, $\mathbf{i}\mapsto \mathbf{i}[t_j]$ is a bijection between $\reduced(w)$ and $\reduced(w^{(j)})$ for any $j$.

Fix $\mathbf{i}=i_1\cdots i_{n-1}\in \reduced(w)$, and let $c=(c_1,\ldots,c_n)\in \wcp{n}$ where $c_i$ is the number of occurrences of $i$ in $\mathbf{i}$.  For the reduced word $\mathbf{i}[t_j]$, the corresponding vector is given by the cyclic shift $c[j]=(c_{t_j+1},\ldots,c_n,c_1,\ldots,c_{t_j})$. By the definition of the indices $t_j$, the $c[j]$ are exactly the cyclic shifts of $c$ that have a nonzero last coordinate. Proposition~\ref{prop:mixed_cyclic} now gives
\[\sum_{j=1}^{k}\frac{A_{c[j]}}{(n-1)!}=1. \]
If we sum the previous identity over all reduced words of $w$, then we obtain \eqref{eq:cyclic_sum} by applying Theorem~\ref{thm:positive_formula} to each term of the previous sum, 
\end{proof}

\begin{example}
Let $w=53124768\in \sgrpp_8$ already considered earlier.
Then one has $|\reduced(w)|=63$ while $a_{w^{(1)}}+a_{w^{(2)}}+a_{w^{(3)}}=6+21+36=63.$
\end{example}

We now present a refined property of the numbers $a_w$ when $w$ is quasiindecomposable, giving a simple way to compute them in terms of principal specializations of Schubert polynomials. Given a permutation $u$ of length $\ell$ and $m\geq 0$, consider
\begin{equation}
\label{eq:nu_m}
\nu_u(m)\coloneqq \nu_{1^m\times u}= \schub_{1^m\times u}(1,1,\ldots).
\end{equation}

By Macdonald's identity \eqref{eq:macdonald} we have
\begin{equation}
\label{eq:macdonald_poly}
\nu_{u}(m)=\frac{1}{\ell!}\sum_{\mathbf{i}\in \reduced(u)}(i_1+m)(i_2+m)\cdots (i_{\ell}+m),
\end{equation}
which is a \emph{polynomial in $m$} of degree $\ell$. Therefore (see \cite{St97} for instance) there exist integers $h_m^u\in \bZ$ for $m=0,\dots,\ell$ such that
 \begin{equation}
\label{eq:defi_h_m}
\sum_{j\geq 0}\nu_u(j)t^j=\frac{\sum_{m=0}^\ell h_m^ut^m}{(1-t)^{\ell+1}}.
\end{equation}
Moreover, the numbers $h_m^u$ are known to sum to $\ell!$ times the leading term of $\nu_{u}(m)$, that is $\sum_{m=0}^\ell h_m^u=|\reduced(u)|$. Thus the following theorem is a refinement of Theorem~\ref{thm:cyclic_sum} in the case of quasiindecomposable permutations.

\begin{theorem}
\label{thm:quasiindecomposable}
Assume that $u\in\sgrp_{p+1}$ is \emph{indecomposable} of length $n-1$.
Define quasiindecomposable permutations $u^{[m]}\in\sgrpp_n$ for $m=0,\ldots,n-p-1$ by $u^{[m]}\coloneqq 1^{m}\times u\times1^{n-p-1-m}$.
Then \begin{align*}
\displaystyle{h_m^{u}=\begin{cases} a_{u^{[m]}}\text{ if }m< n-p;\\ 0\text{ if }m\geq n-p\end{cases}}
\end{align*}
Equivalently, one has
\begin{equation}
\label{eq:summation_indecomposable}
\sum_{j\geq 0}\nu_u(j)t^j=\frac{\sum_{m=0}^{n-p-1} a_{u^{[m]}}t^m}{(1-t)^{n}}.
\end{equation}
\end{theorem}

\begin{proof}
The map $\rho_m:i_1\cdots i_{n-1}\mapsto (i_1+m)\cdots (i_{n-1}+m)$ is a bijection between $\reduced(u)$ and $\reduced(u^{[m]})$ for $m=0,\ldots,n-p-1$.

Fix $\mathbf{i}=i_1\cdots i_{n-1}\in \reduced(u)$. Since $u$ is indecomposable, it has full support, so that $c(\mathbf{i})$ has the form $(a_1,\ldots,a_{p},0,0,\ldots)$ where ${\bf a}=(a_1,\ldots,a_{p})\vDash n-1$. Then $0^m{\bf a}$ is equal to $c(\rho_m(\mathbf{i}))$ for $m=0,\ldots,n-p-1$. We can apply Proposition~\ref{prop:mixed_connected} to ${\bf a}$, and we get:
\[
	\sum_{j\geq 0}(1+j)^{a_1}(2+j)^{a_2}\cdots (p+j)^{a_p}t^j =\frac{\sum_{m=0}^{n-p-1}A_{c(\rho_m(\mathbf{i}))} t^m}{(1-t)^{n}}.
	\]

We now sum this last identity over all $\mathbf{i}\in \reduced(u)$. On the left hand side, for a fixed $j$, the coefficients sum to $(n-1)!\nu_u(j)$ by Macdonald's identity~\eqref{eq:macdonald}. On the right hand side, for a fixed $m$ the coefficients $A_{c(\rho_m(\mathbf{i}))}$ sum to $(n-1)!a_{u^{[m]}}$ by Theorem~\ref{thm:positive_formula}.  This completes the proof of \eqref{eq:summation_indecomposable}.
\end{proof}

\begin{example}
	Consider $n=7$ and $u=4321\in S_4$ an indecomposable permutation.
	We have that $u^{[0]}=4321567$, $u^{[1]}=1543267$, $u^{[2]}=1265437$, and $u^{[3]}=1237654$.
	It is easily checked that
	\[
		\sum_{j\geq 0}\nu_u(j)t^j=\frac{1+7t+7t^2+t^3}{(1-t)^7}
	\]
	Take particular note of the fact that coefficients in the numerator on the right hand side are all positive, which is a priori not immediate.
	Theorem~\ref{thm:quasiindecomposable} then tells us that $a_{u^{[0]}}=1$, $a_{u^{[1]}}=7$, $a_{u^{[2]}}=7$, and $a_{u^{[3]}}=1$. Section~\ref{sec:vexillary} offers a complete explanation for why these numbers arise.
\end{example}

Observe that by extracting coefficients, Theorem~\ref{thm:quasiindecomposable} gives a \emph{signed} formula for $a_{w}$ for any quasiindecomposable $w$ in terms of principal specializations of shifted Schubert polynomials: for any $u\in \sgrp_{p+1}$ indecomposable of length $n-1$, and $m=0,\ldots,n-p-1$, we have that
\begin{equation}
\label{eq:a_from_nu}
a_{u^{[m]}}=\sum_{j=0}^n\nu_u(j)(-1)^{m-j}\binom{n}{m-j}.
\end{equation}

A last observation is that the stability properties from Proposition~\ref{prop:stability} are nicely reflected in Theorem~\ref{thm:quasiindecomposable}. The fact that $a_w=a_{w^{-1}}$ for any $w$ quasiindecomposable is immediate since $\nu_u(j)=\nu_{u^{-1}}(j)$ for any $j$ by \eqref{eq:macdonald_poly}, so that the right hand side of \eqref{thm:quasiindecomposable} for $u$ and $u^{-1}$ coincide.

The stability under $w_o$-conjugation is more interesting: let $\bar{u}=w_o^{p+1}uw_o^{p+1}$ where $w_0^{p+1}$ denotes the longest word in $\sgrp_{p+1}$. Using  \cite[4.2.3]{St97}) we deduce from \eqref{eq:summation_indecomposable} that
\[
\sum_{j\geq 1}\nu_{{u}}(-j)t^j=(-1)^{n-1}\frac{\sum_{m=0}^{n-p-1} a_{u^{[m]}}t^{n-m}}{(1-t)^{n}}
\]

Now $\nu_u(-i)=0$ for $i=1,\ldots,p$ since $u$ has full support, so, using the change of variables $j\mapsto j-p-1$, we can rewrite the previous equation as
\[
\sum_{j\geq 0}\nu_{{u}}(-j-p-1)t^j=(-1)^{n-1}\frac{\sum_{m=0}^{n-p-1} a_{u^{[m]}}t^{n-m-p-1}}{(1-t)^{n}}
\]

We also have $\nu_{\bar{u}}(j)=(-1)^{n-1}\nu_u(-j-p-1)$ easily from \eqref{eq:macdonald_poly}. Putting these together, we get $a_{{\bar{u}}^{[m]}}=a_{{u}^{[n-1-p-m]}}$ for any $m\leq n-p-1$. This is equivalent to
the fact that $a_w=a_{w_oww_o}$ for any $w\in\sgrpp_n$ quasiindecomposable.

\section[Combinatorial interpretation]{Combinatorial interpretation of \texorpdfstring{$a_w$}{} in special cases}
\label{sec:combinatorial_interpretation}

We identify certain special classes of permutations for which we have a combinatorial interpretation. Assume $n\geq 2$ throughout this section.

\subsection{\L{}ukasiewicz permutations}
\label{sub:lukasiewicz}

\begin{definition}
\label{defi:Lukasiewicz_permutations}
A weak composition $(c_1,\ldots,c_n)\in\wcp{n}$ is called \emph{\L{}ukasiewicz} if it satisfies $c_1+\cdots+c_k \geq k$ for any $k\in\{1,\cdots,n-1\}$.

A permutation $w\in \sgrpp_n$ is {\em \L{}ukasiewicz} if $\code{w}$ is a \L{}ukasiewicz composition.
\end{definition}

We note that $c_1+\cdots+c_n=n-1$ since $c$ is assumed to be in $\wcp{n}$, so that the inequality in Definition~\ref{defi:Lukasiewicz_permutations} fails for $k=n$.
Let $\Luk_n$ be the set of \L{}ukasiewicz permutations and $\Lukwc{n}$ the set of \L{}ukasiewicz compositions.
If $Y=\{y_0,y_1,\ldots\}$ is an alphabet, then the words $y_{c_1}y_{c_2}\cdots y_{c_n}$ for $c\in \Lukwc{n}$ are known as \emph{\L{}ukasiewicz words} in $Y$ \cite{Lothaire}.
These are known to be counted by Catalan numbers $Cat_{n-1}=\frac{1}{n}\binom{2n-2}{n-1}$.

\begin{example}
There are $5$ compositions in $\Lukwc{4}$:  \[(3,0,0,0),(2,1,0,0),(2,0,1,0),(1,2,0,0),(1,1,1,0)\] corresponding to the \L{}ukasiewicz permutations $4123,3214,3142,2413,2341$.
\end{example}

\begin{proposition}
\label{prop:Lukasiewicz_permutations_enumeration}
For $n\geq 1$, we have $|\Luk_n|=|\Lukwc{n}|=Cat_{n-1}$.
\end{proposition}

\begin{proof}
We have already argued above that $|\Lukwc{n}|=Cat_{n-1}$. If $c\in\Lukwc{n}$ then $c_i\leq n-i$ for all $i$ since  \[c_i\leq c_i+\ldots+c_n=n-1-(c_1+\cdots+c_{i-1})\leq n-1-(i-1)=n-i.\]
 It follows that the code is a bijection from $\Luk_n$ to $\Lukwc{n}$.
\end{proof}

Our next proposition states that the set of \L{}ukasiewicz permutations is stable under taking inverses.
\begin{proposition}
\label{prop:Lukasiewicz_permutations_inverse}
 If $w\in \Luk_n$ then $w^{-1}\in \Luk_n$.
\end{proposition}
This claim is a priori not clear from the definition, because determining $\mathrm{code}(w^{-1})$ from $\mathrm{code}(w)$ is a convoluted process.
We give a proof based on an alternative characterization of $\Luk_n$ in the appendix.

\subsection[Computation for Lukasiewicz permutations]{Computation of $a_w$ for \L{}ukasiewicz permutations}
We recall Postnikov's result \cite{Pos09} (see also \cite{DS,Pet18}) for the evaluation of divided symmetrization on monomials.
Let $c=(c_1,\ldots,c_n)\in \wcp{n}$.
Define the subset $S_{c}\subseteq{[n-1]}$ by $S_{c}\coloneqq \{k\in[n-1]\suchthat\sum_{i=1}^k c_i<k\}$.
Then
\begin{equation}
\label{eq:monomial_evaluation}
\ds{x_1^{c_1}\cdots x_n^{c_n}}{n}=(-1)^{|S_c|}\beta_n(S_c),
\end{equation}

Here $\beta_n(S)$ is the number of permutations in $\sgrp_n$ with descent set $S$ as defined in Section~\ref{sub:permutation_statistics}.
Recall that we have $a_w=\ds{\schub_w}{n}$, see \eqref{eq:aw_as_ds}, so that by applying~\eqref{eq:monomial_evaluation} to each monomial in the pipe dream expansion \eqref{eq:pipe_dreams} of $\schub_w$, we obtain the formula:
\begin{equation}
\label{eqn:signed_sum_expansion}
 a_w
 =\sum_{\gamma\in \pipe(w)}(-1)^{|S_{c(\gamma)}|}\beta_n(S_{c(\gamma)}).
\end{equation}
In general, this signed sum seems hard to analyze and simplify, and positivity is far from obvious. The nice case where this approach works corresponds precisely to $w\in \Luk_n$.

\begin{theorem}
\label{thm:Lukasiewicz_permutations}
 If $w\in \Luk_n$, then $a_w=|\pipe(w)|$.
\end{theorem}

\begin{proof}
 We examine the expansion \eqref{eq:pipe_dreams} into pipe dreams. If a pipe dream $\gamma$ has weight $(c_1,\ldots,c_n)$, then a ladder move transforms it into a pipe dream $\gamma'$ with weight $(c'_1,\ldots,c'_n)$ where $c'_i=c_{i}+1$, $c'_j=c_{j}-1$ for some $i<j$ while $c'_k=c_k$ for $k\neq i,j$.  In particular $(c_1,\ldots,c_n)\in \Lukwc{n}$ implies $(c'_1,\ldots,c'_n)\in \Lukwc{n}$.

 By definition  the bottom pipe dream $\gamma_{w}$ has weight $\code{w}$ for any $w$. Assume $w\in \Luk_n$ so that the weight of $\Lukwc{n}$ is in $\Lukwc{n}$. It then follows from Theorem~\ref{thm:ladder_moves} that all pipe dreams in the expansion \eqref{eq:pipe_dreams} have weight in $\Lukwc{n}$.

If $(c_1,\ldots,c_n)\in \Lukwc{n}$ then $S_c=\emptyset$ and so $\ds{x_1^{c_1}\cdots x_n^{c_n}}{n}=1$ because $\beta_n(S_c)$ contains only the identity of $\sgrp_n$. Putting things together, we have for any $w\in \Luk_n$,
\[a_w=\ds{\schub_w}{n}=\sum_{\gamma\in \pipe(w)} \ds{x^\gamma}{n}= |\pipe(w)|=\nu_w,\]
which concludes the proof.
\end{proof}

\begin{example}
Let $w=31524\in \Luk_5$ with code $(2,0,2,0,0)$. $\pipe(w)$ consists of $4$ elements, and thus by Theorem \ref{thm:Lukasiewicz_permutations} we get $a_w=4$.
\end{example}

The combinatorial interpretation $a_w=|\pipe(w)|$ shows $a_w>0$ since $\pipe(w)$ contains at least the bottom pipe dream.
By Proposition~\ref{prop:Lukasiewicz_permutations_inverse}, $\Luk_n$ is stable under inverses, and so the stability under taking inverses from Proposition~\ref{prop:stability} is equivalent in this case to $|\pipe(w)|=|\pipe(w^{-1})|$.
This follows combinatorially from the transposition of pipe dreams along the diagonal.

Note that $\Luk_n$ is not stable under conjugation by $w_o$: for instance, for the  permutation $3214$ in $\Luk_4$ we have $w^4_o(3214)w^4_o=1432\notin\Luk_4$. Thanks to Proposition~\ref{prop:stability}, we have
\begin{corollary}
\label{cor:Lukasiewicz_conjugates}
 $a_w=\nu_{w_oww_o}$ if $w_oww_o\in\Luk_n$.
 \end{corollary}
\noindent So for instance we get $a_{1432}=\nu_{3214}=1$. Notice that this is different from $\nu_{1432}=5$.

\begin{remark}
The cardinality $|\Luk_n|=\frac{1}{n}\binom{2n-2}{n-1}$ is asymptotically equal to $4^{n-1}n^{-3/2}/\sqrt{\pi}$ by Stirling's formula. Compared to the asymptotics for $|\sgrpp_n|$ computed in ~\cite{Mar01}, one sees that the ratio $|\Luk_n|/|\sgrpp_n|$ is asymptotically equivalent to $C/n$ for an explicit constant $C$.
\end{remark}

\begin{remark}
\label{rem:dominant}
A \textit{dominant} permutation is defined as a permutation whose code is a partition, or equivalently as a $132$-avoiding permutation \cite{Man01}. Such a permutation has a single pipe dream (necessarily its bottom pipe dream), and so $a_w=1$ by Theorem~\ref{thm:Lukasiewicz_permutations} for any $w\in\sgrpp_n$. By the invariance under $w_o$-conjugation (Corollary~\ref{cor:Lukasiewicz_conjugates})  $213$-avoiding permutations $w$ in $\sgrpp_n$ also satisfy $a_w=1$. Up to $n=11$ these are the only classes of permutations for which $a_w$ is equal to $1$.
\end{remark}

We now connect \L{}ukasiewicz permutations with the cyclic shifts of permutations.

\begin{proposition}
\label{prop:distinct_Lukasiewicz}
For $w\in \sgrpp_n$, the permutations $w^{(i)}$ are pairwise distinct, and exactly one of them is \L{}ukasiewicz.
\end{proposition}
\begin{proof}
Denote by $(c_1,\dots,c_n)$ the code of $w$. All shifts $(c_j,c_j+1,\ldots,c_n,c_1,\ldots,c_{j-1})$ for $j=1,\ldots,n$ are distinct: otherwise $(c_1,\dots,c_n)$ would be periodic which can not be since $n$ and $\sum_ic_i=n-1$ are coprime. Now the {\em cycle lemma} ensures that exactly one of these shifts is in $\Lukwc{n}$; see for instance~\cite[Lemma 9.1.10]{LothaireApplied} with weight $\delta(k)=k-1$. Now these shifts are codes of permutations in $\sgrpp_n$ exactly for the permutations $w^{(i)}$, which completes the proof.
\end{proof}

Notice that as a consequence of Theorems~\ref{thm:Lukasiewicz_permutations} and \ref{thm:cyclic_sum}, we also have the following corollary.
\begin{corollary}
 If $w\in\Luk_n$, then $|\pipe(w)|\leq |\reduced(w)|$.
\end{corollary}

It would be interesting to find a combinatorial proof of this corollary, for instance by finding an explicit injection from $\pipe(w)$ to $\reduced(w)$.

\subsection{Coxeter elements} This case is a subcase of the previous one with particularly nice combinatorics. A \emph{Coxeter element} of $\sgrp_n$ is a permutation that can be written as the product of all elements of the set $\{s_1,s_2,\ldots,s_{n-1}\}$ in a certain order. Let $\Cox_n$ be the set of all Coxeter elements of $\sgrp_n$. Since the defining expressions for Coxeter elements are clearly reduced, we have $\Cox_n\subseteq \sgrpp_n$.

Coxeter elements are naturally indexed by subsets of $[n-2]$ as follows: for $w$ a Coxeter element, define $I_w\subset [n-{2}]$ by the following rule: $i\in I_w$ if and only if $i$ occurs before ${i+1}$ in a reduced word for $w$ (equivalently, in \emph{all} reduced words for $w$). Conversely any subset of $[n-{2}]$ determines a unique Coxeter element, and therefore we have $|\Cox_n|=2^{n-2}$.

\begin{lemma}
$\Cox_n\subseteq \Luk_n$.
\end{lemma}

\begin{proof}
We do this by characterizing codes of Coxeter elements. Let $w\in \Cox_n$, and $I_w=\{i_1<\ldots<i_k\}\subset [n-2]$ as defined above. To $I_w$  corresponds $\alpha_w=(i_1,i_2-i_1,\ldots,i_k-i_{k-1},n-1-{i_k})$ a composition of $n-1$ using a folklore bijection between subsets and compositions.
 Finally, writing $\alpha_w=(\alpha_1,\ldots,\alpha_{k+1}) \vDash n-1$, define the weak composition $c_w$ of $n-1$ with $n$ parts by inserting $\alpha_i-1$ zeros after each $\alpha_i$, and append an extra zero at the end. We claim that $c_w=\code{w}$, leaving the easy verification to the reader.

To illustrate this result, pick $w=2513746\in\Cox_7$, with $431265\in\reduced(w)$. We compute successively $I_w=\{1,4\}\subset [5]$, $\alpha_w=(1,3,2)\vDash 6$ and finally $c_w=(1,3,0,0,2,0,0)$ which is indeed the code of $w$.

An alternative proof is to use Proposition~\ref{prop:alternative_characterization_Lukasiewicz} here: using pipe dreams it is easily shown that $\bar{a}(w)=(1,1,\ldots,1,0)$  for any Coxeter element, and this in fact characterizes such elements. Since $(1,1,\ldots,1,0)\in \Lukwc{n}$, Proposition~\ref{prop:alternative_characterization_Lukasiewicz} ensures that Coxeter elements belong to $\Luk_n$.
\end{proof}

It follows that $a_w=|\pipe(w)|$ if $w\in \Cox_n$ by Theorem~\ref{thm:Lukasiewicz_permutations}.
We note that Sean Griffin~\cite{sean} has managed to give a geometric proof of this fact using Gr\"obner degeneration techniques.

\begin{proposition}
\label{prop:cox_beta}
If $w\in \Cox_n$, then $a_w=\beta_{n-1}(I_w)$.
\end{proposition}

\begin{proof}
It is enough to exhibit a bijection $\phi$ between $\pipe(w)$ and permutations of $\sgrp_{n-1}$ with descent set $I_w$. If $n=2$ then $w=s_1$ and we associate to it the identity permutation in $\sgrp_1$. Now let $w\in \Cox_{n+1}$ for $n\geq 2$.
Note that $\gamma\in \pipe(w)$ has exactly one $+$ in each antidiagonal $A_k$ given by $i+j=k-1$ for $k=1,\ldots,n$; we label them $+_1,\ldots,+_n$. Removing $+_n$  gives a pipe dream $\gamma'$ in $\pipe(w')$  for an element $w'$ in $\Cox_{n}$ since $\gamma'$ has exactly one $+$ in each of the first $n-1$ antidiagonals. By induction we can assume that we have constructed
$\sigma'=\phi(\gamma')\in \sgrp_{n-1}$ with descent set $I_{w'}\subset [n-2]$.

Let $i,j$ be the rows in $\gamma$ containing $+_{n-1},+_n$ respectively. Then define $\sigma$  by incrementing by $1$ all values in $\sigma'$ larger or equal to $n+1-j$, and inserting $n+1-j$ at the end of $\sigma'$.  By immediate induction $\sigma'$ is a permutation ending with $n+1-i$, and $\des(\sigma')=I_{w'}$. Noting that $I_w=I_{w'}\cup\{n-1\}$ if $j>i$ and $I_w=I_{w'}$ if $j\leq i$, one sees that $\des(\sigma)=I_{w}$. We leave the verification that this is a bijection to the reader.
\end{proof}

As interesting special cases, consider the Coxeter elements $w_{odd}$, resp. $w_{even}$, of $\sgrp_n$ defined by the fact that by $I_{w_{odd}}$, resp. $I_{w_{even}}$, consists of all odd, resp. even, integers in $[n-2]$. Then the number $\beta_{n-1}(I_{w_{odd}})=\beta_{n-1}(I_{w_{even}})$ is the \emph{Euler number} $E_{n-1}$ which by definition counts the number of \emph{alternating permutations} in $\sgrp_{n-1}$.
Data up to $n=11$ indicates that the value $a_{w_{odd}}=a_{w_{even}}=E_n$ is the maximal value of $a_w$ over $\sgrpp_n$, and is obtained for these two permutations precisely.

\begin{remark}
Theorem \ref{thm:klyachko} can alternatively be applied directly here to give $a_w=|\reduced(w)|$ instead, since all terms in the sum contribute $1$. The statement of Proposition~\ref{prop:cox_beta} can be deduced from this evaluation also, since reduced words of Coxeter elements are naturally in one-to-one correspondence with standard tableaux of a certain \emph{ribbon shape} attached to $w$, themselves naturally in bijection with permutations having descent set $I_w$. We skip the details.
\end{remark}

\subsection{Grassmannian permutations}
\label{sub:grassmannian}

In this section we give a combinatorial interpretation of $a_w$ when $w$ is a Grassmannian permutation (Theorem~\ref{thm:grassmannian}). Note that this case will be extended to the much larger class of vexillary permutations in Section~\ref{sec:vexillary}.

\begin{definition}
A permutation in $\sgrp_\infty$ is \emph{Grassmannian} if it has a unique descent. It is $m$-Grassmannian if this unique descent is $m\geq 1$.
\end{definition}

The codes $(c_1, c_2,\ldots)$ of $m$-Grassmannian permutations are characterized by $0\leq c_1\leq c_2\leq \cdots\leq c_m$ (with $c_m>0$) while $c_i=0$ for $i>m$.
A Grassmannian permutation $w\in \sgrp_\infty$ is thus encoded by the data $(m,\lambda(w))$, which must satisfy $m\geq \ell(\lambda(w))$.
Conversely any $m,\lambda$ that satisfy $m\geq \ell(\lambda)$ correspond to a permutation in $\sgrp_\infty$. Moreover, such a permutation is in $\sgrp_n$ if and only if $n\geq m+\lambda_1$.

Recall that a \emph{standard Young tableau} $T$ of shape $\lambda\vdash n$ is a filling of the Young diagram of $\lambda$ by the integers $\{1,\ldots,n\}$ that is increasing along rows and columns. A {\em descent} of $T$ is an integer $i<n$ such that $i+1$ occurs in a row strictly below $i$ (here we assume the Young diagram uses the English notation, with weakly decreasing rows from top to bottom).
As illustrated below, for the shape $(3,2)$ for which there are 5 tableaux, the cells containing descents are shaded.

\begin{figure}[!ht]
\includegraphics[scale=1]{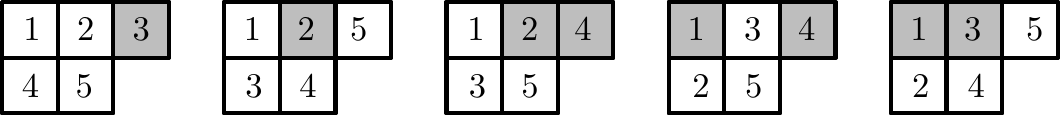}
\end{figure}

Let $\mathrm{SYT}(\lambda)$ be the set of standard Young tableaux of shape $\lambda$ and $\mathrm{SYT}(\lambda,d)$ be the subset thereof containing tableaux with exactly $d$ descents.

\begin{theorem}
\label{thm:grassmannian}
Let $w\in \sgrpp_n$ be a Grassmannian permutation with descent $m$ and shape $\lambda$. Then $a_w$ is equal to $\mathrm{SYT}(\lambda,m-1)$.
\end{theorem}

\begin{proof}
 In this case, the Schubert polynomial $\schub_w$ is known to be the Schur polynomial $s_\lambda(x_1,\cdots,x_m)$ \cite[Proposition 2.6.8]{Man01}.
We thus have to compute $a_w=\ds{s_\lambda(x_1,\ldots,x_m)}{n}$.
This is a consequence of the results of \cite{DS} about divided symmetrizations of (quasi)symmetric functions: see Proposition 4.4 and Example 4.6 in \cite{DS}.
\end{proof}

\begin{example}
Consider the permutations $w_1=351246$ and $w_2=146235$, which are the two Grassmannian permutations in $\sgrpp_6$ with shape $(3,2)$. Note that $w_1$ has descent $2$ while $w_2$ has descent $3$. So $a_{w_1}=\mathrm{SYT}(\lambda,1)=2$ and $a_{w_2}=\mathrm{SYT}(\lambda,1)=3$ from the inspection above.
\end{example}

It is interesting to deduce $a_w>0$ and the invariance under $w_o$-conjugation (cf. Section \ref{sub:positivity}) from this combinatorial interpretation. Note that the inverse of a Grassmannian permutation is not in general Grassmannian, so at this stage the invariance under inverses is not apparent.

Positivity of $a_w$ for $w$ Grassmannian can be shown to be equivalent to the following statement: {\em for any shape $\lambda$ and any integer $d$ satisfying $\lambda_1'-1\leq d\leq |\lambda|-\lambda_1$, then $\mathrm{SYT}(\lambda,d)\neq \emptyset$.} It is indeed possible to construct explicitly such a tableau in $\mathrm{SYT}(\lambda,d)$; we omit the details.

Now suppose $w$ is $m$-Grassmannian with shape $\lambda\vdash n-1$. Then $w_oww_o$ is also Grassmannian, with descent $n-m$ and associated shape $\lambda'$, the transpose of $\lambda$. It is then a simple exercise to show that transposing implies $\mathrm{SYT}(\lambda,m-1)=\mathrm{SYT}(\lambda',n-m-1)$.\smallskip

We finish by giving a pleasant evaluation for a family of mixed Eulerian numbers. Recall that the \emph{content} of a cell in the $i$th row and $j$th column in the Young diagram of a partition $\lambda$ is defined to be $j-i$.
\begin{corollary}
\label{cor:mixed_eulerian_grassmannian}
Let $w\in\sgrpp_{n}$ be an $m$-Grassmannian permutation of shape $\lambda\vdash n-1$.
For $i=1,\ldots,n-1$, let $c_i$ be the number of cells of $\lambda$ with content $i-m$. Then
\[A_{(c_1,\ldots,c_{n-1},0)}=|\mathrm{SYT}(\lambda,m-1)|\prod_{(i,j)\in\lambda}h(i,j),\]
where $h(i,j)=\lambda_i+\lambda'_j-i-j+1$ is the hook-length of the cell $(i,j)$ in $\lambda$.
\end{corollary}

\begin{proof}
Grassmannian permutations are \emph{fully commutative} as they are $321$-avoiding, so all their reduced expressions have the same $c(\mathbf{i})$.
It follows from Theorem~\ref{thm:positive_formula} that
\[a_w=\frac{|\reduced(w)|}{(n-1)!}A_{(c_1,\ldots,c_{n-1},0)}.\]
Now
\[|\reduced(w)|=|\mathrm{SYT}(\lambda)|=\frac{(n-1)!}{\prod_{(i,j)\in\lambda}h(i,j)}\]
by the hook-length formula. The conclusion follows from Theorem~\ref{thm:grassmannian}.
\end{proof}

We discuss the fully commutative case in Section~\ref{sec:further_remarks}.

\section{The case of vexillary permutations}
\label{sec:vexillary}

In this section we will give a combinatorial interpretation to $a_w$ for $w$ vexillary in $\sgrpp_n$.

\begin{definition}
A permutation is \emph{vexillary} it it avoids the pattern 2143.
\end{definition}

They were introduced in \cite{Las82}. This is an important class of permutations in relation to Schubert calculus, containing both dominant and Grassmannian permutations. The Stanley symmetric function $F_w$ \cite{St84} is equal to a single Schur function if and only if $w$ is vexillary. Combinatorially,  vexillary permutations correspond to leaves of the Lascoux-Sch\"utzenberger tree, and play a special role in the Edelman-Greene; see~\cite{Man01} and the references therein.

\begin{proposition}
\label{prop:properties_vexillary}
The class of vexillary permutations in $\sgrp_n$ is closed under taking inverses, and conjugation by $w_o$. Moreover, vexillary permutations are quasiindecomposable.
\end{proposition}

\begin{proof}
Closure under inverses, resp. conjugation by $w_o$, follows immediately from the fact that the pattern 2143 is an involution, resp. is invariant under conjugation by $w_o$.

Now suppose  $w\in \sgrp_n$ is \emph{not} quasiindecomposable.
Then there exist indecomposable $w_i,w_j\neq 1$ with $i<j$ in the factorization \eqref{eq:indecomposables}.
There exists an inversion in each of $w_i,w_j$, and any pair of such inversions give an occurrence of the pattern 2143 in $w$, so that $w$ is not vexillary.
\end{proof}

In particular we will be able to use Theorem~\ref{thm:quasiindecomposable}. We first need to recall certain tableau combinatorics related to vexillary permutations. Then we shall relate these  tableaux to a certain model of {\em $\epsilon$-tableaux}, in order to apply the theory of $(P,\omega)$-partitions to interpret the left hand side of \eqref{eq:summation_indecomposable} in the vexillary case, and ultimately identify the combinatorial interpretation for $a_w$.

\subsection{Flagged tableaux for vexillary permutations}
\label{sub:flagged}

It is known, see~\cite{Las82,Wac85}, that the Schubert polynomials of vexillary permutations are {\em flagged Schur functions}, which we now describe.

Fix a partition $\lambda$ with $l$ parts, and let ${ b}=(b_1,\ldots,b_l)$ be a nondecreasing sequence of positive integers $1\leq b_1\leq\ldots\leq b_l$. A \emph{flagged tableau} $T$ of shape $\lambda$ and flag ${b}$ is a semistandard Young tableau of shape $\lambda$ such that entries in the $i$th row of $T$ lie in $[b_i]$. The weight $\alpx^T$ of $T$ is the monomial $x_1^{m_1} x_2^{m_2}\cdots $ with $m_i$ the number of entries $i$ in $T$.
Let $\mathrm{SSYT}(\lambda;{ b})$ be the set of flagged tableaux of shape $\lambda$ and flag ${ b}$.
Then \[s_\lambda(\alpx;{ b})=\sum_{T\in \mathrm{SSYT}(\lambda;{ b})}\alpx^T\] is the corresponding {\em flagged Schur function}.\smallskip

Now let $w\in \sgrp_\infty$ be a partition with code ${ c}=\code{w}$. Recall that the shape $\lambda(w)$ is the partition obtained by sorting the nonzero entries of $ c$ in nonincreasing order. Given $i$ such that $c_i>0$, define $e_i$ to be the maximal $j$ such that $c_j\geq c_i$. The {\em flag} $\phi(w)$ of $w$ is defined by ordering the $e_i$ in nondecreasing order.

This can be expressed in a more compact way as follows: Write $\lambda$ uniquely in the  form $\lambda=(p_1^{m_1},p_2^{m_2},\dots,p_r^{m_r})$ with
$ p_1>p_2>\cdots>p_r$. For $1\leq q\leq r$, let $\phi_q$ be the maximum index $j$ such that $c_j\geq p_q$. Then it is clear that $\phi(w)=(\phi_1^{m_1},\ldots,\phi_r^{m_r})$.
\begin{example}
	Consider $w=812697354\in S_9$. We have $\code{w}=(7, 0, 0, 3, 4, 3, 0, 1, 0)$.
	We compute $e_1=1$, $e_4=6$, $e_5=5$, $e_6=6$ and $e_8=8$.
	Thus $\phi(w)=(1,5,6,6,8)$.

	Alternatively, express $\lambda(w)=(7,4,3^2,1)$. We have $\phi_1=1$, $\phi_2=5$, $\phi_3=6$, and $\phi_4=8$.
	This gives the same flag as before.
\end{example}
We note further that an $m$-Grassmannian permutation has flag $\phi=(m,\ldots,m)$, while a dominant permutation has flag $\phi=(m_1^{m_1},(m_1+m_2)^{m_2},\dots,(m_1+m_2+\cdots+m_r)^{m_r})$.

If $w$ is vexillary of shape $\lambda(w)$, then $\schub_w= s_{\lambda(w)}(\alpx,\phi(w))$ (cf. \cite{Las82,Wac85}) and in particular \[\nu_w=|\mathrm{SSYT}(\lambda(w),\phi(w))|.\]

\begin{proposition}\cite{Las82,Macdonald}
\label{prop:vexillary_characterization}
A vexillary permutation is characterized by the data of its shape and flag.
Moreover, $(\lambda=(p_1^{m_1},\ldots,p_r^{m_r}),\phi=(\phi_1^{m_1},\ldots,\phi_r^{m_r}))$ is equal to $(\lambda(w),\phi(w))$ for $w$ vexillary if and only if the following inequalities are satisfied:
\begin{align}
\label{eq:first_set}\phi_q&\geq m_1+\cdots+m_q\quad&\text{for  } q=1,\ldots,r;\\
\label{eq:second_set} 0\leq \phi_{q+1}-\phi_{q}&\leq m_{q+1}+p_{q}-p_{q+1}\quad&\text{for  } q=1,\ldots,r-1.
\end{align}
\end{proposition}

The first set of inequalities is easy to prove (and valid for any permutation). The second one is more involved, cf. \cite{Macdonald}. It is interesting to consider the extreme cases of each:
\begin{itemize}
\item $\phi_q=m_1+\cdots+m_q$ for $q=1,\ldots,r$ iff $w$ is dominant.
\item $\phi_q=\phi_{q+1}$ for $q=1,\ldots,r-1$  iff $w$ is Grassmannian.
\item $\phi_{q+1}-\phi_{q}=m_{q+1}+p_{q}-p_{q+1}$ for $q=1,\ldots,r-1$ iff $w$ is inverse Grassmannian, that is $w^{-1}$ is Grassmannian.
\end{itemize}

\subsection{Plane partitions with arbitrary strict conditions on rows and columns}
\label{sub:epsilon}

We fix $\lambda=(\lambda_1,\dots,\lambda_l)$, where $l=\lambda_1'$ is the number of parts. Recall that a plane partition of shape $\lambda$ is an assignment $T_{i,j}\in\{0,1,2,\ldots\}$ for $(i,j)\in \lambda$ that is weakly decreasing along rows and columns. In other words, if $P_\lambda$ is the poset of cells of $\lambda$ in which $c\leq c'$ if $c$ is to the northwest of $c'$, then a plane partition of shape $\lambda$ is a $P_\lambda$-partition in the sense of Stanley \cite[Section 4.5]{St97}.

\begin{definition}
 A \textit{signature} for $\lambda$ is an ordered pair $\epsilon=(e,f)\in \{0,1\}^{l-1}  \times \{0,1\}^{\lambda_1-1}$.

 An \emph{$\epsilon$-partition} of shape $\lambda$ is a plane partition $(T_{i,j})$ of shape $\lambda$ such that for all $j$, $T_{i,j}>T_{i+1,j}$ if $e_i=1$,  and for all $i$,  $T_{i,j}>T_{i,j+1}$ if $f_j=1$.
\end{definition}

Thus, in an $\epsilon$-partition entries must strictly decrease between rows (\emph{resp.} columns) $i$ and $i+1$ if $e_i=1$ (\emph{resp.} $f_i=1$). Let $\Omega(\lambda,\epsilon,N)$ be the number of $\epsilon$-partitions of shape $\lambda$ with maximal entry at most $N$.
An example of $\epsilon$-partition is given in Figure~\ref{fig:epsilon-partition} for $N=6$.
Plane partitions correspond to the signature $e_i=f_j=0$ for all $i$ and $j$.\smallskip

A \emph{labeling} $\omega$ of $P_\lambda$ is a bijection from $P_\lambda$ to $\{1,\dots,|\lambda|\}$.
Let $\omega_\epsilon$ be a \emph{compatible labeling}: that is, it satisfies $\omega_\epsilon(i,j)>\omega_\epsilon(i+1,j)$ if and only if $e_i=1$, and $\omega_\epsilon(i,j)>\omega_\epsilon(i,j+1)$ if and only if $f_j=1$.

Such a labeling always exists: indeed, let $G_{\lambda,\epsilon}$ be the directed graph whose underlying undirected graph is the Hasse diagram of $P_\lambda$, and with orientation given by $(i,j)\rightarrow (i,j+1)$ if and only if $e_i=1$, and $(i,j)\rightarrow (i+1,j)$ if and only if $f_j=1$. The orientation is easily seen to be {\em acyclic}, which ensures the existence of compatible labelings $\omega_\epsilon$ since those are precisely the \emph{topological orderings} of $G_{\lambda,\epsilon}$, that is the linear orderings of its vertices such that if $u\rightarrow v$ then $\omega_\epsilon(u)< \omega_\epsilon(v)$. These exist exactly when the graph is a directed acyclic graph (DAG).

\begin{figure}[!ht]
\includegraphics{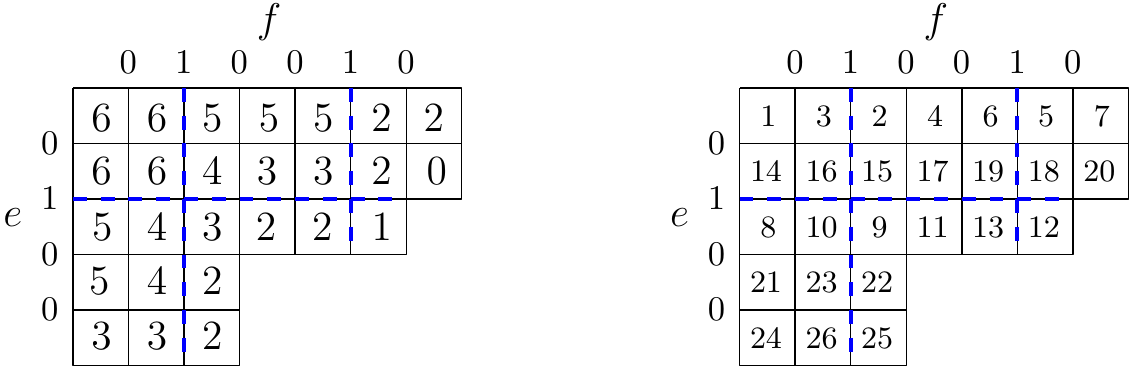}
\caption{ $\lambda=(7,7,6,3,3)$ with signature $\epsilon=(0100,010010)$. An $\epsilon$-partition (left) and a compatible labeling $\omega_\epsilon$ (right). \label{fig:epsilon-partition}
}
\end{figure}

We now recognize that an $\epsilon$-partition of shape $\lambda$ is precisely a $(P_\lambda,\omega_\epsilon)$-partition \cite[Section 7.19]{St99}. By the general theory of $(P,\omega)$-partitions, we get the following result: Let $ SYT(\lambda)$ be the set of standard tableaux of shape $\lambda$. An \textit{$\omega_{\epsilon}$-descent} of $T\in \mathrm{SYT}(\lambda)$ is an entry $k<|\lambda|$ such that $\omega_{\epsilon}(T^{-1}(k))>\omega_{\epsilon}(T^{-1}(k+1))$. Let $\dsc(T;w_\epsilon)$ be the number of $\omega_{\epsilon}$-descents of $T$. Then
\begin{equation}
\label{eq:partitions_to_descents}
\sum_{N\geq 0} \Omega(\lambda,\epsilon,N) t^N=\frac{\sum_{T\in \mathrm{SYT}(\lambda)}t^{\dsc(T;w_\epsilon)}}{(1-t)^{|\lambda|+1}}.
\end{equation}

\subsection{From $\epsilon$-tableaux to flagged tableaux.}
Fix $\lambda,\epsilon$ as in the previous section. We will see that $\Omega(\lambda,\epsilon,N)$ naturally enumerates flagged semistandard tableaux. By taking complements $T_{i,j}\mapsto N+1-T_{i,j}$, we have that $\Omega(\lambda,\epsilon,N)$ counts \textit{$\epsilon$-tableaux}, defined as fillings of $\lambda$ with integers in $\{1,\ldots,N+1\}$ weakly increasing in rows and columns, with strict increases forced  by $e,f$.
Let $\mathcal{T}(\lambda,\epsilon,N)$ be the set of $\epsilon$-tableaux with entries at most $N+1$; by definition $|\mathcal{T}(\lambda,\epsilon,N)|=\Omega(\lambda,\epsilon,N)$. \smallskip

Write $\lambda=(p_1^{m_1}>p_2^{m_2}>\cdots>p_r^{m_r})$ as before, and define $M_q=m_1+\cdots+m_q$ for $q=1\ldots,r$.
Define the partial sums \[\begin{cases}
E_i=E_i(\epsilon)\coloneqq\sum_{k=1}^{i-1} e_k\text{ for }i=1,\ldots, l,\\
F_j=F_j(\epsilon)\coloneqq\sum_{k=1}^{j-1}f_k\text{ for }j=1,\ldots, \lambda_1.
\end{cases}
\]
Also consider $\bar{E}_i=i-1-E_i$ and $\bar{F}_j=j-1-F_j$.
We remark that $\mathcal{T}(\lambda,\epsilon,N)\neq \emptyset$ if and only if
\begin{equation}
\label{eq:Nmin}
N\geq F_{p_q}+E_{M_q}\text{  for  } q=1,\ldots, r.
\end{equation}
Informally put, the quantity $F_{p_q}+E_{M_q}$ counts the number of strict increases that are forced in going from the top left  cell of $\lambda$ to the corner cell in column $p_q$.
For the $\epsilon$-tableau on the left in Figure~\ref{fig:epsilon_to_flagged}, the $E$ and $F$ vectors are given by $(0,0,1,1,1)$ and $(0,0,1,1,1,2,2)$ respectively, and their barred analogues are given by $(0,1,1,2,3)$ and $(0,1,1,2,3,3,4)$.\smallskip

 We want to transform tableaux in $\mathcal{T}(\lambda,\epsilon,N)$ into semistandard Young tableaux, that is $(1^{l-1},0^{\lambda_1-1})$-tableaux.  The general idea is to decrease values in the columns to the right of a strict condition $f_j=1$, and to increase the values in the rows below a weak condition $e_i=0$. This leads to the following definition.

\begin{definition}
\label{defi:str}
Fix an $\epsilon$-tableau $T \in\mathcal{T}(\lambda,\epsilon,N)$. We define $\str(T)=T'$ to be the filling of $\lambda$ given by
\[T'_{i,j}=T_{i,j}-F_j+\bar{E}_i\quad\text{for all }(i,j)\in \lambda.\]
\end{definition}

The $\epsilon$-tableau on the left in Figure~\ref{fig:epsilon_to_flagged}
belongs to $\mathcal{T}(\lambda,\epsilon,N)$ for $\lambda=(7,7,6,3,3)$, $\epsilon=(0100,010010)$, and $N=6$.
Its image under $\str$ is depicted on the right using the $E$ and $F$ computed earlier.
Proposition~\ref{prop:bijection_str} states that $\str$ is bijective between $\mathcal{T}(\lambda,\epsilon,7)$ and $\mathrm{SSYT}(\lambda;(6^2,6^1,9^2))$.
\begin{figure}[!ht]
\includegraphics{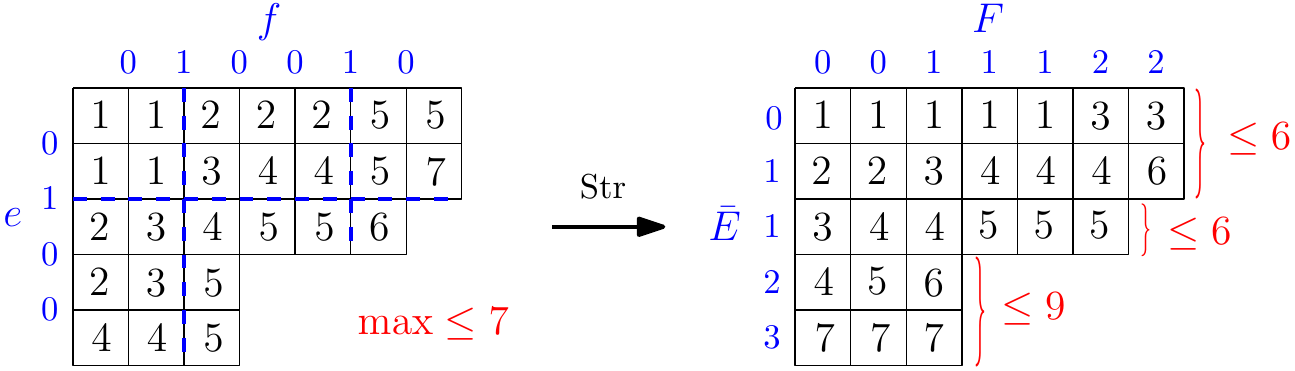}
\caption{\label{fig:epsilon_to_flagged} The $\epsilon$-tableau coming from the $\epsilon$-partition of Figure~\ref{fig:epsilon-partition} (left), and its image under $\str$ (right). The bounds in red indicate constraints of tableaux for which $\str$ is bijective, cf. Proposition~\ref{prop:bijection_str}.
}
\end{figure}

It is easily checked that $T'=\str(T)$ is a semistandard Young tableau. Indeed checking that the columns of $T'$ are strictly increasing amounts to showing that $e_i<T_{i+1,j}-T_{i,j}+1$, whereas showing that the rows are weakly decreasing is equivalent to $f_j\leq T_{i,j+1}-T_{i,j}$. Both these inequalities are immediate.
We now work out what the condition that the maximal entry  in $T$ is at most $N+1$ becomes under the mapping $\str$.

Define $\phi_{\epsilon,N}\coloneqq (\phi_1^{m_1},\ldots,\phi_r^{m_r})$ by
\begin{equation}
\label{eq:defi_phi}
\phi_q=N+1-F_{p_q}+\bar{E}_{M_q}
\end{equation}
 for $q=1\ldots,r$.
 It follows that for $1\leq q\leq r-1$,
 \begin{align}
 \label{eq:conditions_phi}
 \delta_q\coloneqq\phi_{q+1}-\phi_{q}=(\bar{E}_{M_{q+1}}-\bar{E}_{M_{q}})+(F_{p_{q}}-F_{p_{q+1}})
 \end{align}
 is equal to the number of zeros in $e$ between rows $M_{q}$ and $M_{q+1}$ plus the number of ones in $f$ between columns $p_{q+1}$ and $p_{q}$. Therefore $\phi_{\epsilon,N}$ satisfies the inequalities ~\eqref{eq:second_set}.
 \smallskip

 Furthermore, the inequalities \eqref{eq:Nmin} become $\phi_q\geq 1+ E_{M_q}+\bar{E}_{M_q}=M_q$ for $q\geq 1$, which is precisely the inequalities ~\eqref{eq:first_set}.
We invite the reader to check that in our running example, we have that $\phi_1=7-2+1$, $\phi_2=7-2+1$,  and $\phi_3=7-1+3$.
 	This means that $\phi_{\epsilon,N}=(6^2,6^1,9^2)$.

\begin{proposition}
\label{prop:bijection_str}
 Given $\epsilon$ and $N$ satisfying~\eqref{eq:Nmin}, $(\lambda,\phi_{\epsilon,N})$ corresponds to a vexillary permutation $w$.  Furthermore, $\str$ is a bijection between $\mathcal{T}(\lambda,\epsilon,N)$ and $\mathrm{SSYT}(\lambda,\phi_{\epsilon,N})$.
\end{proposition}

\begin{proof}
We have already checked that the inequalities of Proposition~\ref{prop:vexillary_characterization} were satisfied under the hypotheses. It is also clear that $\str$ is well-defined, and that
$U_{i,j}\mapsto U_{i,j}+F_j-\bar{E}_i$ provides the desired inverse.
\end{proof}

\subsection{Combinatorial interpretation of $a_w$}
Let $w$ be a vexillary permutation of shape $\lambda\vdash n-1$ and flag $\phi$. From Proposition \ref{prop:properties_vexillary}, $w=1^m \times u$ with $u$ indecomposable and vexillary. Clearly $\lambda(u)=\lambda$, while  $\phi(w)$ is obtained from $\phi(u)$ by adding $m$ to each entry; let us write this $\phi(w)=m+\phi(u)$ in short. We thus have
\begin{equation}
\nu_u(m)=|\mathrm{SSYT}(\lambda,m+\phi(u))|.
\end{equation}

The next lemma provides some converse to Proposition~\ref{prop:bijection_str}.

\begin{lemma}
\label{lem:signature_from_vexillary}
Let $u$ be indecomposable and vexillary. There exists a signature $\epsilon_u$ on $\lambda(u)$ and a nonnegative integer $N_u$ such that $\phi(u)=\phi_{\epsilon_u,N_u}$. Moreover $N_u$ is given by \[N_u=\max_q(F_{p_q}(\epsilon_u)+E_{M_q}(\epsilon_u)).\]
\end{lemma}

\begin{proof}
	Let $\phi\coloneqq \phi(u)$, $\lambda\coloneqq\lambda(u)$.
		Also, like before $l=\ell(\lambda)$. We claim that there exist $(e_1,\dots,e_{l-1})\in \{0,1\}^{l-1}$ and $(f_1,\dots,f_{\lambda_1-1})\in\{0,1\}^{\lambda_1-1}$ such that
		\begin{align}\label{eqn:ef}
		\sum_{M_{q}\leq i\leq M_{q+1}-1}(1-e_i)+\sum_{p_{q+1}\leq j\leq p_{q}-1}f_j=\phi_{q+1}-\phi_{q}
		\end{align}
		has solutions for all $1\leq q\leq r-1$.
		Indeed, as $u$ is vexillary, the inequalities \eqref{eq:second_set} state that for any $1\leq q\leq r-1$, we have $\phi_{q+1}-\phi_{q}\leq m_{q+1}+p_{q}-p_{q+1}$.
		Now, in \eqref{eqn:ef}, the first sum runs over $m_{q+1}$ elements, whereas the second sum runs over $p_{q}-p_{q+1}$ elements.
		It therefore follows that we can pick  $e_{M_{q}},\ldots,e_{M_{q+1}-1}, f_{p_{q+1}},\ldots,f_{p_{q}-1}$ in $\{0,1\}$ such that \eqref{eqn:ef} is satisfied. In fact, there are in general many such choices.
		Having made these choices for $1\leq q\leq r-1$, we subsequently pick $e_{1},\ldots,e_{M_1-1}, f_{1},\ldots,f_{p_1-1}$ arbitrary to obtain $(e_1,\dots,e_{l-1})$ and $(f_1,\dots,f_{\lambda_1-1})$.

These choices comprise our signature $\epsilon_u$. Indeed, it is readily checked that \eqref{eqn:ef} is \eqref{eq:conditions_phi} in disguise.
Now define $\phi'=\phi_{\epsilon_u,N_u}$ with the value of $N_u$ in the lemma.
There is thus an equality in~\eqref{eq:Nmin} for a certain $q\in [r]$, which translates to an equality in \eqref{eq:first_set} for the same $q$. This shows that the vexillary permutation determined by the flag $\phi'$ does not have $1$ as a fixed point.
It is therefore equal to $u$, and it follows that $\phi'=\phi$ as wanted.
\end{proof}

\begin{example}
\label{ex:signature_from_vexillary}
	Consider $u=346215$ with shape $\lambda=(3^1,2^2,1^1)$ and $\phi(u)=(3^1,3^2,4^1)$.
	We then have $(p_1,p_2,p_3)=(3,2,1)$ and $(M_1,M_2,M_3)=(1,3,4)$.
	The sequences $(e_1,e_2,e_3)$ and $(f_1,f_2)$ which comprise the signature $\epsilon_u$ need to satisfy  $(1-e_3)+f_1=1$ and $(1-e_2)+(1-e_1)+f_2=0$.
	Thus, we may pick $(e_1,e_2,e_3)=(1,1,0)$, and $(f_1,f_2)=(0,0)$. The corresponding $E$ and $F$ vectors are therefore $(0,1,2,2)$ and $(0,0,0)$ respectively.
	It follows that $N_u$ is $\max{\{0+0,0+2,0+2\}}=2$.
\end{example}

\begin{theorem}
\label{thm:vexillary}
Let $u\in \sgrp_{p+1}$ of shape $\lambda\vdash n-1$ be an indecomposable vexillary permutation, and choose $\epsilon_u,N_u$ as in Lemma~\ref{lem:signature_from_vexillary}. Moreover, let $\omega_u\coloneqq \omega_{\epsilon_u}$ be an $\epsilon_u$-compatible labeling as defined in Section~\ref{sub:epsilon}.

Let $m\in\{0,\ldots,n-p-1\}$ and consider the permutation $u^{[m]}\in\sgrpp_n$ defined by $u^{[m]}=1^m\times u\times1^{n-p-1}$. Then we have \[\sum_{j\geq 0}\nu_u(j)t^j=\frac{\sum_{T\in \mathrm{SYT}(\lambda)}t^{\dsc(T;\omega_u)-N_u}}{(1-t)^{n}}.\]
\end{theorem}

\begin{proof}
We have \[\nu_u(j)=|\mathrm{SSYT}(\lambda;j+\phi(u))|=|\mathrm{SSYT}(\lambda;j+\phi_{\epsilon_u,N_u})|=|\mathrm{SSYT}(\lambda;\phi_{\epsilon_u,j+N_u})|,\] and so by Proposition~\ref{prop:bijection_str} we get \[\nu_u(j)=|\mathcal{T}(\lambda;\epsilon_u,j+N_u)|=\Omega(\lambda;\epsilon_u,j+N_u),\]
and therefore
\[\sum_{j\geq 0}\nu_u(j)t^j=\sum_{j\geq 0}\Omega(\lambda;\epsilon_u,j+N_u)t^j=t^{-N_u}\sum_{j\geq 0}\Omega(\lambda;\epsilon_u,j)t^j,\]
because $\Omega(\lambda;\epsilon_u,j)=0$ for $j<N_u$.
From \eqref{eq:partitions_to_descents} the desired identity follows.
\end{proof}

Comparing the content of Theorem~\ref{thm:vexillary} with \eqref{eq:summation_indecomposable} from Theorem~\ref{thm:quasiindecomposable} gives the following as an immediate corollary:

\begin{corollary}
\label{cor:vexillary}
We keep the notations from Theorem~\ref{thm:vexillary}. Then $a_{u^{[m]}}$ is equal to the number of tableaux $T\in \mathrm{SYT}(\lambda)$ with $m+N_u$ $\omega_{\epsilon_u}$-descents.
\end{corollary}

\begin{example}
We follow up on Example~\ref{ex:signature_from_vexillary}. The next figure depicts a possible $\omega_{\epsilon_u}$.
	$$
	\ytableausetup{mathmode,boxsize=1em}
	\begin{ytableau}
		5& 6 & 7\\ 3 & 4\\1 & 2\\8
	\end{ytableau}
	$$
	Here are the three standard Young tableaux with exactly two $\omega_{\epsilon}$-descents, coming from the shaded boxes.
	$$
	\begin{ytableau}
	1& *(gray!50)2 & 7\\ 3 & *(gray!50)4\\5 & 6\\8
	\end{ytableau}\hspace{10mm}
	\begin{ytableau}
	1& *(gray!50)2 & *(gray!50)5\\ 3 & 4\\6 & 7\\8
	\end{ytableau}\hspace{10mm}
	\begin{ytableau}
	1& 2 & *(gray!50)3\\ 4 &*(gray!50) 5\\6 & 7\\8
	\end{ytableau}
	$$
	It follows that $a_{u^{[0]}}=a_{346215789}=3$. The reader may further verify that
	\[
		\sum_{j\geq 0}\nu_u(j)t^j=\frac{3+24t+34t^2+9}{(1-t)^{9}}.
	\]

	To further demonstrate that we have a family of combinatorial interpretations depending on the choice of $\epsilon_u$ (and $\omega_u$), an alternative legitimate choice for $u=346215$ is the signature $(1,1,1),(1,0)$, for which $N_u$ equals $\max{\{1+0,1+2,0+3\}}=3$.
	Suppose we pick $\omega_{u}$ to read $738\hspace{2mm} 62\hspace{2mm}51 \hspace{2mm}4$ going top to bottom, left to right in the Young diagram of shape $\lambda$. Here are the three tableaux $\mathrm{SYT}(\lambda)$ with exactly three $\omega_{u}$-descents.
	$$
	\begin{ytableau}
	*(gray!50)1& 2 & 8\\ *(gray!50)3 & 4\\*(gray!50)5 & 6\\7
	\end{ytableau}\hspace{10mm}
	\begin{ytableau}
	*(gray!50)1& 2 & *(gray!50)6\\ *(gray!50)3 & 4\\5 & 7\\8
	\end{ytableau}\hspace{10mm}
	\begin{ytableau}
	*(gray!50)1& 2 & *(gray!50)4\\ *(gray!50)3 & 5\\6 & 7\\8
	\end{ytableau}
	$$
\end{example}
Let us revisit the Grassmannian and dominant cases in light of our treatment of the vexillary case.
We borrow notation that we have used throughout this section.
\begin{enumerate}
	\item If $u$ is indecomposable Grassmannian, then the signature $\phi\coloneqq \phi(u)$ satisfies $\phi_q-\phi_{q-1}=0$. 	It follows that we may pick $(e_1,\dots,e_{l-1})=(1^{l-1})$ and $(f_1,\dots,f_{\lambda_1-1})=(0^{\lambda_1-1})$.
	If we pick $\omega_{\epsilon}$ to correspond to the  filling of $\lambda\coloneqq \lambda(w)$ where we place integers from $1$ through $|\lambda|$ from bottom to top and left to right, we see that an $\omega_{\epsilon}$-descent is the same as a traditional descent in $\mathrm{SYT}$, thereby recovering Theorem~\ref{thm:grassmannian}.
	\item Next consider $u$ dominant.  One can see that $(e_1,\dots,e_{l-1})=(0^{l-1})$ and $(f_1,\dots,f_{\lambda_1-1})=(0^{\lambda_1-1})$ give a valid signature. We pick the natural labeling where we place integers from $1$ through $|\lambda|$ from  top to bottom and left to right, so that an $\omega_{\epsilon}$-descent is a traditional ascent of an SYT.
\end{enumerate}

We remark that {\em shifted dominant} permutations of the type $1\times u$ for $u$ dominant occur in a number of articles \cite{Ber18,Esc16,woo04}.\smallskip

Finally, let us briefly sketch why the invariance properties of Proposition~\ref{prop:stability} are apparent in this combinatorial interpretation.
Fix $\lambda\vdash n-1$, and let $H_{q}\coloneqq m_{q+1}+p_{q}-p_{q+1}$ for $q=1,\ldots,r-1$ using previously introduced notation.
Let $u\in \sgrp_{p+1}$ be an indecomposable vexillary with shape $\lambda$ and flag differences $\delta_q\coloneqq\phi_{q+1}-\phi_{q}$ for ${q=1,\ldots,r-1}$.
Define $\bar{u}=w_o^{p+1}uw_o^{p+1}$ where $w_0^{p+1}$ denotes the longest word in $\sgrp_{p+1}$. Then it follows from \cite[Formulas (1.41) and (1.42)]{Macdonald} that the indecomposable vexillary permutations $\bar{u}$ and $u^{-1}$ are characterized as follows:
\begin{itemize}
\item $\bar{u}$ has shape $\lambda'$ and flag differences  $(\delta_{r-q})_{q=1,\ldots,r-1}$;
\item $u^{-1}$ has shape $\lambda'$ and flag differences $(H_{r-q}-\delta_{r-q})_{q=1,\ldots,r-1}$.
\end{itemize}
We fix a signature $\epsilon_u=(e,f)$ and a labeling $\omega_u$ for $u$ as in Theorem~\ref{thm:vexillary}. Then the following claims are easily checked:
\begin{itemize}
\item A valid signature for $\bar{u}$ is given by $\epsilon_{\bar{u}}\coloneqq (f,e)$ on $\lambda'$. A compatible $\omega_{\bar{u}}$ is defined by $\omega_{\bar{u}}(i,j)\coloneqq\omega_{u}(j,i)$ for any $(i,j)\in\lambda'$.
\item  A valid signature for $u^{-1}$ is given by $\epsilon_{\bar{u}}\coloneqq (1-f,1-e)$ on $\lambda'$ where naturally $(1-f)_j=1-f_j$ and $(1-e)_i=1-e_i$. A compatible $\omega_{u^{-1}}$ is defined by $\omega_{u^{-1}}=n-\omega_{\bar{u}}$.
\end{itemize}
We leave it to the interested reader to show the invariance properties of Proposition~\ref{prop:stability} from the combinatorial interpretation afforded by Corollary~\ref{cor:vexillary} (the invariance under conjugation by $w_o$ is more involved).

\section{Klyachko's original formula and $a_w$ for other types}
\label{sec:klyachko_formula}

While the majority of this article is concerned with type $A$, we now deal with any Lie type $\Phi$. We want to describe the class of the permutahedral variety in type $\Phi$ in terms of certain mixed $\Phi$-Eulerian numbers in a manner akin to Theorem~\ref{thm:positive_formula}.

The starting point is again Klyachko's work \cite{Kly85,Kly95}. We state and give  Klyachko's beautiful ``Macdonald-like formula''\footnote{Compare equation~\eqref{eq:macdonald} and the equality in Theorem~\ref{th:KM}; see \cite{NT21}.} which was first announced in ~\cite{Kly85}, and appeared with a proof some time later in~\cite{Kly95}. Since the latter is in Russian, and for the sake of completeness, we reproduce  Klyachko's proof here with some slight improvement. \medskip

\subsection{Klyachko's theorems}

Fix $G$ a complex connected reductive group, $B$ a Borel subgroup and $T$ a maximal torus inside $B$. Let $\Phi$ be the root system of rank $r$, and $W$ be the Weyl group $W\coloneqq N_G(T)/T$. Let $\Delta=\{\alpha_1,\dots,\alpha_{r}\}$ denote the set of simple roots, $\Pi$ the corresponding set of positive roots. Recall that $\Pi$ is in one-to-one correspondence with the set of reflections of $W$, which we note $\alpha\mapsto s_\alpha$. 
We denote by $\langle \cdot,\cdot\rangle$ the Killing form. We say that $i$ is a descent of $w\in W$ if $\ell(ws_{\alpha_i})=\ell(w)-1$, and let $\des(w)$ be the set of descents of $w$.

The cohomology ring $H^*(G/B,\bQ)$ has a basis given by Schubert classes $\sigma_w$ as $w$ ranges over elements in $W$. Denote by $X=X(\Phi)\subset G/B$ the closure of a generic orbit of the maximal torus $T \subset G$: $X$ is \emph{the permutahedral variety of type $\Phi$}. It is a smooth projective variety of dimension $r$. It can be constructed alternatively as the toric variety attached to the Coxeter fan of type $\Phi$.

Consider the algebra homomorphism $i^*:H^*(G/B, \bQ) \to H^*(X,\bQ)$ induced from the inclusion $X=X\subset G/B$ . Klyachko \cite{Kly85,Kly95} shows that the image of  $i^*$ coincides with the algebra of invariants $H^*(X,\bQ)^W$, and gives a presentation for this algebra as follows:  Denote by $\mathcal{L}_{\Lambda}$ the line bundle on $G/B$ induced by a weight $\Lambda$, that is, a character $\Lambda:B\to \bC^*$. Let $[\Lambda] =  c_1(\mathcal{L}_{\Lambda}|_X)\in H^2(X,\bQ)$ be the first Chern class of the restriction of $L_\Lambda$ to $X$. 
 Finally denote by $\Lambda_i, i=1,\dots r$ the fundamental weights of $\mathfrak{g}$.

\begin{theorem}[\cite{Kly85,Kly95}]
\label{th:invariant_cohomology}
The algebra $H^*(X, \bQ)^{W}$ is generated by the classes $[\Lambda_i]$, $i=1,\ldots,r$ subject only to the
quadratic relations
\begin{equation}
[\Lambda_i][\alpha_i]=0,  \text{ for } 1\leq i\leq n-1.
\end{equation}
It has dimension $2^r$, with basis given by the squarefree monomials in the generators $[\Lambda_i]$.
\end{theorem}

In  type $A$ this recovers the presentation for $\D_n$ given in Section~\ref{sub:Klyachko}, by writing the roots in terms of fundamental weights.\smallskip

Given $w\in W$, let $\mathrm{Red}(w)$ denote the set of reduced words for $w$: $i_1\dots i_m\in\mathrm{Red}(w)$  if $s_{i_1}\dots s_{i_m}$ is a reduced expression for $w$.
The next result describes the image of the Schubert class $\sigma_w$.  
 
\begin{theorem}[\cite{Kly85,Kly95}]
\label{th:KM}
For $w\in W$, we have the identity in $H^*(X, \bQ)^{W}$
\begin{equation}
\label{eq:KM}
i^*(\sigma_w)=\frac{1}{\ell(w)!}\sum_{i_1\dots i_{\ell(w)}\in \mathrm{Red}(w)}[\Lambda_{i_1}]\cdots [\Lambda_{i_{\ell(w)}}].
\end{equation}
\end{theorem}

Klyachko establishes this result by verifying that both sides satisfy the same recursion. We now give a simplified version of this argument \footnote{Klyachko's proof uses galleries between any two chambers in the Coxeter arrangement. It is actually enough to consider reduced expressions of $w$, that is minimal galleries starting from the fundamental chamber.}. We need a couple of preliminary results:

\begin{lemma}
\label{lem:easy?}
For any $w\in W$, define 
\begin{align*}
A_w&=\{(\alpha,\beta)\suchthat \alpha\in \des(w),\beta\in \Pi\setminus \{\alpha\}, \ell(ws_{\beta}s_{\alpha})=\ell(w)\},\\
B_w&=\{(\alpha,\beta)\suchthat \beta\in \Pi,\alpha\in \des(ws_{\beta}), \ell(ws_{\beta})=\ell(w)+1\}.
\end{align*}
Then $A_w\subset B_w$ and $B_w\setminus A_w=\{(\alpha,\alpha)\suchthat \alpha\notin \des(w)\}$.
\end{lemma}
\begin{proof}
This follows from standard arguments in Coxeter theory, see \cite[Lemma 2.4]{BGG73} for a proof.
\end{proof}

Let $P(w)=\displaystyle{\sum_{i_1\dots i_{\ell(w)}\in \mathrm{Red}(w)}[\Lambda_{i_1}]\cdots [\Lambda_{i_{\ell(w)}}]}$, the sum in the right hand side in \eqref{eq:KM}.
\begin{proposition}
\label{prop:induction}
For any weight $\Lambda$ and any $w\in W$, 
we have in $H^*(X, \bQ)^{W}$
\[[\Lambda][P(w)]=\frac{1}{(\ell(w)+1)}\sum_{\substack{\beta\in \Pi\\ \ell(ws_{\beta})=\ell(w)+1}}\langle \check{\beta},\Lambda\rangle [P(ws_{\beta})]
,\]

where $\check{\beta}$ denotes the coroot attached to $\beta$. 
\end{proposition}

\begin{proof}
Write $u_i=[\Lambda_i]$. We proceed by induction on $\ell(w)$. The case $w=e$ corresponds to the equality
\[
[\Lambda]=\sum_{i=1}^r \langle\check{\alpha}_i, \Lambda\rangle u_i
\]
which holds because of the expansion $\Lambda=\sum_{i=1}^r\langle\check{\alpha}_i, \Lambda\rangle \Lambda_i$. Now if $\ell(w)>0$, we have:
\begin{align*}
[\Lambda][P(w)]&=\sum_{i\in\des(w)}[\Lambda]u_i[P(ws_i)]=\sum_{i\in\des(w)}[s_i\Lambda]u_i[P(ws_i)],
\end{align*}
the first equality follows by splitting according to the last letter of the reduced expression, and the second one from the relations $u_i[\alpha_i]=0$.

By induction, we obtain the following sequence of equalities (see explanation below):
\begin{align*}
\ell(w)[\Lambda][P(w)]&=
\sum_{\substack{i\in\des(w),\gamma\in \Pi\\ \ell(ws_is_{\gamma})=\ell(w)}}\langle\check{\gamma},s_i\Lambda\rangle u_i[P(ws_is_\gamma)]\\
&=\sum_{\substack{i\in\des(w),\beta\in s_i\Pi\\ \ell(ws_{\beta}s_i)=\ell(w)}}\langle\check{\beta},\Lambda\rangle u_i[P(ws_\beta s_i)]\\
&=\sum_{\substack{i\in\des(w),\beta\in \Pi-\{\alpha_i\}\\ \ell(ws_{\beta}s_i)=\ell(w)}}\langle\check{\beta},\Lambda\rangle u_i[P(ws_\beta s_i)]-\sum_{i\in\des(w)}u_i\langle\check{\alpha}_i, \Lambda\rangle [P(w)]\\
&=\sum_{\substack{\beta\in \Pi, i\in\des(ws_\beta),\\ \ell(ws_{\beta})=\ell(w)+1}}\langle\check{\beta},\Lambda\rangle u_i[P(ws_\beta s_i)]-\sum_{i=1}u_i\langle\check{\alpha}_i, \Lambda\rangle [P(w)].
\end{align*}

The first equality applies induction to $[s_i\Lambda][P(ws_i)]$ for each $i$, the second one is a change of variables $\beta=s_i(\gamma)$, the third follows from the decomposition $s_i\Pi=\Pi-\{\alpha_i\}\sqcup \{-\alpha_i\}$, and the last one is Lemma~\ref{lem:easy?}.
\end{proof}

\begin{proof}[Proof of Theorem~\ref{th:KM}]  The Schubert classes $\sigma_w$ are known to satisfy 

\[c_1(L_\Lambda)\cup\sigma_w=\sum_{\substack{\beta\in \Pi\\ \ell(ws_{\beta})=\ell(w)+1}}\langle \check{\beta},\Lambda\rangle \sigma_{ws_{\beta}}
,\]
 in $H^*(G/B,\bQ)$ for any $w
\in W$,  cf. \cite{GelSer87}. It follows that $\ell(w)!i^{*}(\sigma_w), w\in W$ satisfy the recursion of Proposition \ref{prop:induction}. It thus remains to check the initial conditions $P(s_i)=i^*(\sigma_{s_i})(=[\Lambda_i])$ for all $i$, which is immediate.
\end{proof}

\subsection{Application}
\label{sub:Application_}

Let $a_w^\Phi$ the coefficients of $[X(\Phi)]$ in the cohomology of the generalized flag variety $G/B$:\[[X(\Phi)]=\sum_{w\in W'} a_w^\Phi \sigma_{w_ow},\] where $W'\subset W$ consist of the elements of length $r$. These naturally extend the numbers $a_w$ to all types, and are nonnegative numbers since they compute intersections as in type $A$. 

Given a weak composition $c=(c_1,\dots,c_r)$ of $r$, let $A_{c}^{\Phi}$ denote the \emph{mixed $\Phi$-Eulerian numbers} indexed by $c$, introduced by Postnikov \cite[Definition 18.4]{Pos09}. Like the mixed Eulerian numbers introduced earlier, the $A_{c}^{\Phi}$ are defined to be mixed volumes of $\Phi$-hypersimplices; equivalently, they occur as coefficients in the expansion of the volume plynomial of the type $\Phi$-permutahedron. For a combinatorial description of these numbers in type $B$, the reader is referred to \cite{Liu16}.

 By the relation between volumes and degrees $A_{c}^{\Phi}$  can be computed as a certain mixed degree (up to a factor $r!$), which, pushing the computation to $H^*(X, \bQ)^{W}$ via the morphism $i^*$, gives the following explicit rule: write $u_i=[\Lambda_i]$ for the generators of $H^*(X, \bQ)^{W}$ as in Theorem~\ref{th:invariant_cohomology}. Then $A_{c}^{\Phi}$ is the coefficient of $u_1\dots u_r/r!$ in the squarefree basis expansion of $u_1^{c_1}\dots u_r^{c_r}$.
 
Now as in type $A$, $a_w^\Phi$ can be computed in $H^*(X, \bQ)^{W}$ as the coefficient of the fundamental class of $X$ in the expansion of $i^*(\sigma_w)$ in $H^*(X, \bQ)^{W}$.  Thus one can finally extract the coefficient of $u_1\cdots u_r$ for $w\in W'$ in \eqref{eq:KM} to obtain the following generalization of Theorem~\ref{thm:positive_formula}.
\begin{theorem} \label{th:aw_in_other_types} For any $w\in W$ of length $r$ and $\mathbf{i}\in \reduced(w)$, let $c(\mathbf{i})=(c_1,\ldots,c_{r})$ where $c_j$ counts occurrences of$j$ in $\mathbf{i}$. Then\begin{equation}\label{eq:positive_formula_again}a_w^{\Phi}=\sum_{\mathbf{i}\in \reduced(w)}\frac{A_{c(\mathbf{i})}^{\Phi}}{r!}.\end{equation} \end{theorem} 
All $A_c^{\Phi}$ are positive integers because of their definition as mixed volumes. It thus follows
\begin{corollary}
For any $w\in W'$, $a_w^\Phi$ is positive and satisfies $a_{w^{-1}}^\Phi=a_w^\Phi$.
\end{corollary}

The positivity of $a_w^\Phi$ solves the problem briefly considered by Harada et al. \cite[Remark 6.7]{HarHor18}. 
It would be interesting to undertake a combinatorial study of the $a_{w}^{\Phi}$ outside of type $A$ as well.

\section{Further remarks}
\label{sec:further_remarks}

\subsection{}
The original motivation for this paper was to investigate a combinatorial interpretation for the numbers $a_w$. We know from geometry that the numbers $a_w$ are nonnegative, can we find a family of objects counted by $a_w$? This was achieved in this work for \L{}ukasiewicz permutations (Theorem \ref{thm:Lukasiewicz_permutations}) and vexillary permutations (Theorem~\ref{thm:vexillary}).

The hope is to find a combinatorial interpretation in general, from which the various properties established in Section~\ref{sec:properties} would be apparent. Note that Theorem \ref{thm:cyclic_sum} strongly suggests that $a_w$ counts a subset of the reduced words of $w$, which in turn hints that the Edelman-Greene correspondence \cite{EG87} may play a role.

Based on Theorem \ref{thm:cyclic_sum}, it would be interesting to generalize the results in Section~\ref{sec:vexillary} to encompass the whole class of quasiindecomposable permutations.

A natural special case, which generalizes the Grassmannian case, is when $w$ is quasiindecomposable and fully commutative. Since the number of reduced words $\mathbf{i}$ for such a $w$ is the number of $\mathrm{SYTs}$ $f^{\lambda/\mu}$ for an appropriate connected skew shape $\lambda/\mu$ with $n-1$ boxes, and all such $\mathbf{i}$ give the same $c(\mathbf{i})$, the question of giving a combinatorial interpretation for $a_w$ amounts to giving one for $\frac{f^{\lambda/\mu}}{(n-1)!}A_{c(\mathbf{i})}$. Also the Schubert polynomial in this case is a flagged skew Schur function, so that $\nu_w$ can be interpreted as counting certain flagged skew tableaux; an approach in the manner of Section~\ref{sec:vexillary} may be successful.
As a curious aside, we remark here that one can derive the hook-content formula for $\lambda/\mu$ by piecing together our  Theorem~\ref{thm:quasiindecomposable}, Theorem~\ref{thm:positive_formula}, and Proposition~\ref{prop:mixed_connected}.

\subsection{}
Theorems \ref{thm:cyclic_sum} and \ref{thm:quasiindecomposable} give pleasant summation formulas for the numbers $a_w$. It would be interesting to find a common generalization of them. We note that Theorem~\ref{thm:quasiindecomposable} fails in general: in fact, our data seems to show that as soon as $u$ is not indecomposable, the numerator on the right hand side has at least one negative coefficient.

Another avenue worth exploring, and more in line with the theme of \cite{Ber20} and motivated by Brenti's Poset Conjecture \cite{Bre89}, is investigating aspects like real-rootedness, unimodality and log-concavity for the numerators of the right hand side in Theorem~\ref{thm:quasiindecomposable}.
By work of Brenti \cite{Bre89} and Br\"and\'en \cite{Br04,Br04b}, the Grassmannian case is already well understood.

\subsection{}
Given $w\in \sgrp_\infty$, consider the polynomial $\Mt_w(x_1,x_2,\ldots)$ defined by
\[
\Mt_w\coloneqq\frac{1}{\ell(w)!}\M_w(x_1,x_1+x_2,x_1+x_2+x_3,\ldots)=
\frac{1}{\ell(w)!} \sum_{\mathbf{i}\in \reduced(w)}y^{c(\mathbf{i})}.
\]

 Now let $w\in\sgrpp_n$. It is quite striking to compare the formulas given by the two approaches of Section~\ref{sec:formulas}. Indeed by Macdonald's identity \eqref{eq:macdonald}, we have~$\Mt_{w}(1,1,\ldots)=\schub_{w}(1,1,\ldots)=\nu_w$. Also, by Theorems~\ref{thm:anderson_tymoczko} and \ref{thm:positive_formula}, we moreover have $\ds{\Mt_{w}}{n}=\ds{\schub_{w}}{n}=a_w$. The coincidence between these specializations is a reflection of a phenomenon explored in greater generality in \cite{NT21}.
 
 \subsection{}
The summatory results for connected mixed Eulerian numbers (Proposition \ref{prop:mixed_connected}) and  quasiindecomposable permutations (Theorem~\ref{thm:quasiindecomposable}) can be expressed compactly in terms of certain {\em back stable} analogues, inspired by the work of Lam, Lee and Shimozono~\cite{Lam18}.

Consider the algebra $\mathcal{B}$ of bounded degree power series in $\bQ[[x_i,i\in\mathbb{Z}]]$ that are polynomials in the $x_i,i>0$, and symmetric in the $x_i,i\leq 0$.
Thus $\mathcal{B}$ identifies naturally with $\Lambda(x_i,i\leq 0)\otimes \bQ[x_i,i>0]$.
Let $f\in \mathcal{B}$ be homogeneous of degree $n-1$, written $f\in\mathcal{B}^{(n-1)}$.
Following \cite{Lam18}, consider the truncation operator $\pi_{+}(f)\coloneqq f(\ldots,0,x_{1},x_{2},\ldots)$ and the shift operator $\gamma$ that sends $x_i\mapsto x_{i+1}$ for all $i\in \mathbb{Z}$.
This given, define $f[m]\coloneqq \pi_{+}(\gamma^m(f))$ which is a polynomial in $x_1,x_2,\ldots$, and let $f[m](\mathbf{1})$ denote its evaluation when all $x_i,i>0$ are specialized to $1$.
Then $f[m](\mathbf{1})$ is a polynomial in $m$ of degree $\leq n-1$ (easy), and we infer the existence of $h_m^{f}\in \bQ$ such that
\begin{equation}
\sum_{j\geq 0} f[j](\mathbf{1}) t^j=\frac{\sum_{m\geq 0}h_m^{f}t^m}{(1-t)^n}.
\end{equation}
\begin{definition}
Let $\mathcal{D}^n$ be the subspace of $f\in\mathcal{B}^{(n-1)}$ such that $h_m^{f}=\ds{f[m]}{n}$ for any $m\geq 0$.
\end{definition}
	We now briefly touch upon some elements that lie $\mathcal{D}^n$ by our results.
First, Theorem~\ref{thm:quasiindecomposable} says that the \emph{back stable Schubert polynomial} $\overleftarrow{\schub_u}$ \cite{Lam18} is in $\mathcal{D}^n$ if $u$ is indecomposable of length $n-1$.

Additionally, if $f$ is a symmetric function in the $x_i,i<0$, then $f[m]$ is the symmetric polynomial $f(x_1,\ldots,x_m)$. The fact that $f\in\mathcal{D}^n$ is one of the main results of~\cite{DS}.

Let $\overleftarrow{y_k}$ be the series $\overleftarrow{y_k}=\ldots +x_{-2}+x_{-1}+x_{0}+\ldots+x_{k-1}+x_{k}=\sum_{i\leq k} x_i.$
Given ${\bf a}\in \wc{p}{n-1}$, define $\overleftarrow{y_{\bf a}}=\overleftarrow{y_1}^{a_1}\overleftarrow{y_2}^{a_2}\cdots \overleftarrow{y_p}^{a_p}$.
Then Proposition~\ref{prop:mixed_connected} says precisely that if ${\bf a}$ is a strong composition, that is ${\bf a}\vDash n-1$, then $\overleftarrow{y_{\bf a}}\in \mathcal{D}^n$.

In view of the aforementioned, the following problem is natural: \emph{Characterize the space $\mathcal{D}^n$, for instance by finding a distinguished basis.}
By working in an `infinite' version of $\D_n$, we obtain a partial answer to this question in \cite{NT21}.

\subsection{}\label{subsec:approach_AT} By expanding a double Schubert polynomial in terms of Schubert polynomials (cf. \cite{Man01}), Formula~ \eqref{eq:regular_class_doubleschubert} gives
\begin{equation}
\label{eq:regular_class_expansion}
\Sigma_h=\sum_{\substack{u,v \in {S}_n \\ v^{-1}u=w_{h} \\ \ell(u)+\ell(v)=\ell(w_{h})}}\schub_u\schub_v\mod I_n.
\end{equation}

In~\cite{And10}, this latter formula is used to give an explicit expansion of $\Sigma_h$ in the Schubert basis in the easy special case where $w_h\in \sgrp_k\subset \sgrp_n$ with $2k\leq n$.
Recall that $w_h$ is the permutation given by $\code{w_h^{-1}}=(n-h(1),\ldots,n-h(n))$.

In the case $h=(2,3,\ldots,n,n)$ that is the subject of our study, we have $w_h=w_o^{n-1}$, so we get
\[\tau_n=\sum_{\substack{u,v\in \sgrp_n\\uv^{-1}=w_o^{n-1} \\ \ell(u)+\ell(v)=\binom{n-1}{2}
}} \sigma_u\sigma_{w_ovw_o}\]

We may simplify the summation range: as shown in \cite[Lemma 6.1]{HarHor18}, the conditions are equivalent to $u\in\sgrp_{n-1}$ (and $v=w_o^{n-1}u$).
Let us give a short proof:  For any $u\in\sgrp_n$, $\ell(u)+\ell(w_o^{n-1}u)\geq \ell(w_o^{n-1})=\binom{n-1}{2}$, since any pair $(i,j)$ with $1\leq i<j\leq n-1$ is an inversion in either $u$ or $w_o^{n-1}u$. It follows then that $\ell(u)+\ell(w_o^{n-1}u)=\ell(w_o^{n-1})$ if no pair $(i,n)$ is an inversion either $u$ or $w_o^{n-1}u$, which is clearly equivalent to $u(n)=n$ so that $u\in \sgrp_{n-1}$. Therefore we can write
\[ \class_n=\sum_{u \in\sgrp_{n-1}
} \sigma_u\sigma_{1\times w_o^{n-1}u}.\]

Extracting coefficients gives the summation formulas for $w\in\sgrpp_n$:
\begin{equation}
\label{eq:aw_cuvw}
a_w=\sum_{u \in\sgrp_{n-1}}c^w_{u,1\times w_o^{n-1}u},
\end{equation}
where the structure coefficients $c_{u,v}^w$ are defined in
\eqref{eq:structure_coefficients_schubert_2}. Together with the combinatorial interpretations (Theorem~\ref{thm:Lukasiewicz_permutations}, Corollary~\ref{cor:vexillary}) and our various other results about the $a_w$, Equation \eqref{eq:aw_cuvw} gives information about certain coefficients $c_{uv}^w$ that may be of interest in the quest to find a combinatorial interpretation for them.

\subsection{} To go beyond the focus of this work, a natural endeavour is to compute the coefficients in the Schubert basis for the other regular Hessenberg classes $\Sigma_h$, see Section~\ref{sub:AT}.

As mentioned above, this was essentially done in \cite{And10} for the case $w_h\in \sgrp_k\subset \sgrp_n$ with $2k\leq n$; they also consider the case where $h(i)=n$ for $i>1$.
The starting point is the formula~\eqref{eq:regular_class_doubleschubert} for $\Sigma_h$.

Let us also mention the work \cite{Ins20} which gives another polynomial representative for $\Sigma_h$: consider the permutation  $w_h'\in \sgrp_{2n}$ given by $w_h'(i+h(i))=n+i$ for $i\in [n]$ and put the values $1,\ldots,n$ from left to right in the remaining entries. Then
\begin{equation}
\label{eq:regular_class_bigschubert}
\Sigma_h=\schub_{w'_h}(x_1,\ldots,x_{h(1)},x_1,x_{h(1)+1},\ldots,x_{h(2)},x_2,x_{h(2)+1},\ldots,x_{h(n)},x_n)\mod I_n
\end{equation}

We would also like to emphasize the recent work of Kim~\cite{Kim20}: he investigates a larger family of cohomology classes, in all types, coming from varieties related to the Deligne-Lusztig varieties. His formulas in type $A$ extend those of~\cite{And10}.

\appendix

\section{Proof of Proposition \ref{prop:Lukasiewicz_permutations_inverse}}

Let $w\in \sgrp_n$ with code $(c_1,\dots,c_{n-1})$. We define the composition $\bar{a}(w)=(a_1,\dots,a_n)$ by
 \begin{equation}
 \label{eq:a_from_c}
  a_i=|\{1\leq j\leq i\suchthat c_j> i-j\}|.
  \end{equation}
  More generally, consider $\gamma\in \pipe(w)$. Following~\cite{Wei18}, let ${a}(\gamma)=(a_1,a_2,\ldots,a_n)$ where $a_k$ is the number of $+$'s on the $k$th antidiagonal $i+j=k-1$. Then $\bar{a}(w)=a(\gamma_w)$ where $\gamma_w$ is the bottom pipe dream of $w$.

 \begin{example}  For $w=153264$ we have $\code{w}=(0,3,1,0,1,0)$ and $\bar{a}(w)=(0,1,2,1,1,0)$, while if $w=413265$, then $\code{w}=(3,0,1,0,1,0)$ and $\bar{a}(w)=(1,1,2,0,1,0)$.
   For the first permutation, neither $\code{w}=(0,3,1,0,1,0)$ nor $\bar{a}(w)$ are in $\Lukwc{n}$, while both of them are in $\Lukwc{n}$ in the second case. Refer to the diagram that follows.
  \begin{figure}[!ht]
 \includegraphics[scale=0.8]{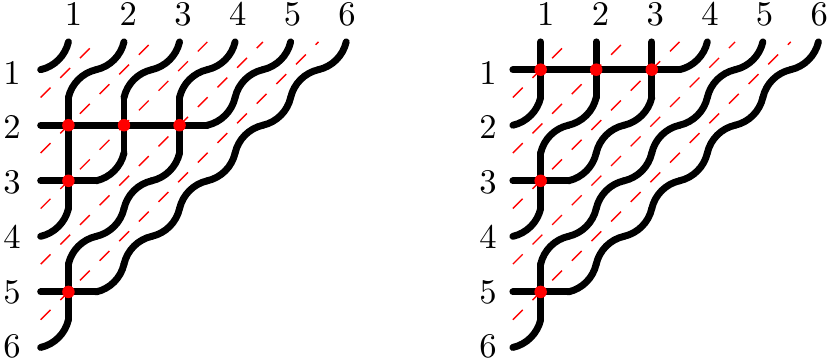}
 \end{figure}
 \end{example}
  \begin{proposition}\label{prop:alternative_characterization_Lukasiewicz}
   For $w\in \sgrpp_n$, we have that $\code{w}\in \Lukwc{n}$  if and only if $\bar{a}(w)\in \Lukwc{n}$.
 \end{proposition}

 \begin{proof}
 Write $\code{w}=(c_1,\ldots,c_n)$ and $\bar{a}(w)=(a_1,\dots,a_n)$. For $1\leq i\leq n-1$, we have
     \begin{align}\label{eqn:sum_aj_geq_j}
       \sum_{1\leq j\leq i}a_j=\sum_{1\leq j\leq i}\min\{c_j,i-j+1\}\leq \sum_{1\leq j\leq i}c_j.
     \end{align}
It follows immediately that if $\bar{a}(w)\in \Lukwc{n}$ then $c(w)\in \Lukwc{n}$.

Conversely, assume $\bar{a}(w)\notin \Lukwc{n}$, so that there exists $1\leq k\leq n-1$ such that
     \begin{align}
       \sum_{1\leq j\leq k}a_j<k.
     \end{align}
Let $k$ be the smallest integer with this property. This forces $\sum_{1\leq j\leq k-1}a_j=k-1$ and $a_k=0$ (note that this holds in the special case $k=1$ also). By \eqref{eq:a_from_c} this implies in turn that $c_j\leq k-j$ for $j=1,\ldots,k$ and thus, by using the leftmost equality in \eqref{eqn:sum_aj_geq_j},
     \begin{align}
     \sum_{1\leq j\leq k}c_j=\sum_{1\leq j\leq k}a_j=k-1.
     \end{align}
      Therefore $\code{w}\notin \Lukwc{n}$, which finishes the proof.
 \end{proof}

\begin{proof}[Proof of Proposition~\ref{prop:Lukasiewicz_permutations_inverse}]
We use here~\cite[Lemma 3.6(iii)]{Wei18}  which states that for any $w\in\sgrp_\infty$, ${a}(\gamma_w)={a}(\gamma_{w^{-1}})$, which translates into $\bar{a}(w)=\bar{a}(w^{-1})$. We then conclude by Proposition~\ref{prop:alternative_characterization_Lukasiewicz}.
\end{proof}

\section*{Acknowledgments}
We are grateful to Sara Billey, Sean Griffin, and particularly Alex Woo for numerous helpful discussions at the early stages of this project.
We would also like to thank Marcelo Aguiar, Dave Anderson, Jang Soo Kim, Alejandro Morales, Allen Knutson, and Alex Yong for enlightening discussions.

\section*{Tables}
\begin{table}[!ht]
\centering
\caption{The Schubert class expansions $\displaystyle{\tau_n=\sum_{w\in\sgrpp_n}a_w\sigma_{w_ow}}$ for $ 2\leq n\leq 6$. Indexing permutations $w_ow$ are highlighted if $w$ is not quasiindecomposable. } \label{tab:schub_exp}
\begin{tabular}{ |p{1cm}|p{14cm} | }
	\hline
	$n$ & Schubert expansions for $\tau_n$\\
\hline
$2$ & $\schub_{1}$\\
\hline
$ 3 $ & $\schub_{132} + \schub_{21}$ \\
\hline
$ 4 $ & $ \schub_{1 4 3 2} + \schub_{2 3 4 1} + 2\schub_{2 4 1 3} + 2\schub_{3 1 4 2} + \schub_{3 2 1} + \schub_{4 1 2 3} $ \\
\hline
$5 $ & $ \schub_{1 5 4 3 2} + \schub_{2 4 5 3 1} + 2\schub_{2 5 3 4 1} + 3\schub_{2 5 4 1 3} + \schub_{3 2 5 4 1} + \schub_{3 4 2 5 1} + 4\schub_{\textcolor{magenta}{3 4 5 1 2}} + 5\schub_{3 5 1 4 2} + 3\schub_{3 5 2 1 4} + 3\schub_{4 1 5 3 2} + 2\schub_{4 2 3 5 1} + 5\schub_{4 2 5 1 3} + 3\schub_{4 3 1 5 2} + \schub_{4 3 2 1} + 4\schub_{\textcolor{magenta}{4 5 1 2 3}} + 2\schub_{5 1 3 4 2} + \schub_{5 1 4 2 3} + \schub_{5 2 1 4 3} + 2\schub_{5 2 3 1 4} + \schub_{5 3 1 2 4} $\\
\hline
$ 6$ & $ \schub_{1 6 5 4 3 2} + \schub_{2 5 6 4 3 1} + 2\schub_{2 6 4 5 3 1} + 3\schub_{2 6 5 3 4 1} + 4\schub_{2 6 5 4 1 3} + \schub_{3 4 6 5 2 1} + \schub_{3 5 4 6 2 1} + 3\schub_{3 5 6 2 4 1} + 5\schub_{\textcolor{magenta}{3 5 6 4 1 2}} + 3\schub_{3 6 2 5 4 1} + 3\schub_{3 6 4 2 5 1} + 10\schub_{\textcolor{magenta}{3 6 4 5 1 2}} + 9\schub_{3 6 5 1 4 2} + 6\schub_{3 6 5 2 1 4} + 2\schub_{4 2 6 5 3 1} + \schub_{4 3 5 6 2 1} + 2\schub_{4 3 6 2 5 1} + 5\schub_{\textcolor{magenta}{4 3 6 5 1 2}} + 3\schub_{4 5 2 6 3 1} + \schub_{4 5 3 2 6 1} + 5\schub_{\textcolor{magenta}{4 5 3 6 1 2}} + 10\schub_{\textcolor{magenta}{4 5 6 1 3 2}} + 10\schub_{\textcolor{magenta}{4 5 6 2 1 3}} + 9\schub_{4 6 1 5 3 2} + 8\schub_{4 6 2 3 5 1} + 16\schub_{4 6 2 5 1 3} + 11\schub_{4 6 3 1 5 2} + 4\schub_{4 6 3 2 1 5} + 10\schub_{\textcolor{magenta}{4 6 5 1 2 3}} + 4\schub_{5 1 6 4 3 2} + 3\schub_{5 2 4 6 3 1} + 8\schub_{5 2 6 3 4 1} + 11\schub_{5 2 6 4 1 3} + 3\schub_{5 3 2 6 4 1} + 2\schub_{5 3 4 2 6 1} + 10\schub_{\textcolor{magenta}{5 3 4 6 1 2}} + 16\schub_{5 3 6 1 4 2} + 9\schub_{5 3 6 2 1 4} + 6\schub_{5 4 1 6 3 2} + 3\schub_{5 4 2 3 6 1} + 9\schub_{5 4 2 6 1 3} + 4\schub_{5 4 3 1 6 2} + \schub_{5 4 3 2 1} + 10\schub_{\textcolor{magenta}{5 4 6 1 2 3}} + 10\schub_{\textcolor{magenta}{5 6 1 3 4 2}} + 5\schub_{\textcolor{magenta}{5 6 1 4 2 3}} + 5\schub_{\textcolor{magenta}{5 6 2 1 4 3}} + 10\schub_{\textcolor{magenta}{5 6 2 3 1 4}} + 5\schub_{\textcolor{magenta}{5 6 3 1 2 4}} + 3\schub_{6 1 4 5 3 2} + 2\schub_{6 1 5 3 4 2} + \schub_{6 1 5 4 2 3} + 3\schub_{6 2 3 5 4 1} + 4\schub_{6 2 4 3 5 1} + 8\schub_{6 2 4 5 1 3} + 2\schub_{6 2 5 1 4 3} + 3\schub_{6 2 5 3 1 4} + 3\schub_{6 3 1 5 4 2} + 3\schub_{6 3 2 4 5 1} + 3\schub_{6 3 2 5 1 4} + 8\schub_{6 3 4 1 5 2} + 3\schub_{6 3 4 2 1 5} + 3\schub_{6 3 5 1 2 4} + 3\schub_{6 4 1 3 5 2} + 3\schub_{6 4 1 5 2 3} + 2\schub_{6 4 2 1 5 3} + 2\schub_{6 4 2 3 1 5} + \schub_{6 4 3 1 2 5} + \schub_{6 5 1 2 4 3} + \schub_{6 5 1 3 2 4} + \schub_{6 5 2 1 3 4} $ \\
\hline
\end{tabular}
\end{table}

\begin{longtable}[h]{| c |c| c |}
\caption{The numerator on the right hand side of \eqref{eq:summation_indecomposable} for indecomposable $u$ (up to inverses and conjugation by the longest word).  The non-vexillary $u$ are highlighted. \label{tab:indecomposable}}\\
	\hline
  Indecomposable $u\in \sgrp_{p+1}$ & $\ell(u)$ & $ \displaystyle{\sum_{m=0}^{\ell(u)-p} a_{u^{[m]}}t^m}$ \\
 \hline
 $21$ & $1$ & $1$\\
 \hline\hline
 $231$ & $2$ & $1$\\
 \hline
 $321$ & $3$ & $t+1$\\
 \hline\hline
 $2 3 4 1$ & $3$ & $ 1$\\
 $ 2 4 1 3$ & $3$ & $ 2$\\
 \hline
$ 2 4 3 1$ & $4$ & $t+2$\\
$ 3 4 1 2$ & $4$ & $ t+1$\\
\hline
$ 4 2 3 1$ & $5$ & $ t^2+4\,t+1$\\
\hline
$ 4 3 2 1$ & $6$ & $t^3+7\,t^2+7\,t+1 $\\
\hline
\hline
$ 2 3 4 5 1 $ & $4$ & $ 1 $ \\
$ 2 3 5 1 4 $ & $4$ & $ 3 $ \\
$\textcolor{magenta}{2 4 1 5 3} $ & $4$ & $ 5 $ \\
\hline
$ 2 3 5 4 1 $ & $5$ & $ t + 3 $ \\
$ 2 4 3 5 1 $ & $5$ & $ 2 \, t + 2 $ \\
$ 2 4 5 1 3 $ &$5$ & $ 2 \, t + 3 $ \\
$ \textcolor{magenta}{2 5 1 4 3} $ & $5$ & $ 3 \, t + 8 $ \\
$ 2 5 3 1 4 $ & $5$ & $ 3 \, t + 3 $ \\
\hline
$ 2 4 5 3 1 $ & $6$ & $ t^{2} + 5 \, t + 3 $ \\
$ 2 5 3 4 1 $ & $6$ & $ t^{2} + 6 \, t + 3 $ \\
$ 2 5 4 1 3 $ & $6$ & $ 2 \, t^{2} + 9 \, t + 5 $ \\
$\textcolor{magenta}{ 3 2 5 4 1 }$ & $6$ & $ 3 \, t^{2} + 13 \, t + 3 $ \\
$ 3 4 5 1 2 $ & $6$ & $ t^{2} + 3 \, t + 1 $ \\
$ 3 5 1 4 2 $ & $6$ & $ 2 \, t^{2} + 10 \, t + 4 $ \\
\hline
$ 2 5 4 3 1 $ & $7$ & $ t^{3} + 11 \, t^{2} + 18 \, t + 5 $ \\
$ 3 4 5 2 1 $ & $7$ & $ t^{3} + 6 \, t^{2} + 6 \, t + 1 $ \\
$ 3 5 4 1 2 $ & $7$ & $ t^{3} + 8 \, t^{2} + 10 \, t + 2 $ \\
$ 5 2 3 4 1 $ & $7$ & $ t^{3} + 9 \, t^{2} + 9 \, t + 1 $ \\
\hline
$ 3 5 4 2 1 $ & $8$ & $ t^{4} + 14 \, t^{3} + 34 \, t^{2} + 19 \, t + 2 $ \\
$ 4 5 3 1 2 $ & $8$ & $ t^{4} + 10 \, t^{3} + 20 \, t^{2} + 10 \, t + 1 $ \\
$ 5 2 4 3 1 $ & $8$ & $ t^{4} + 17 \, t^{3} + 45 \, t^{2} + 25 \, t + 2 $ \\
\hline
$ 5 4 3 2 1 $ & $10$ & $ t^{6} + 31 \, t^{5} + 187 \, t^{4} + 330 \, t^{3} + 187 \, t^{2} + 31 \, t + 1 $ \\
\hline
\end{longtable}

\bibliographystyle{hplain}
\bibliography{Biblio_Peterson}

\end{document}